\documentclass[]{scrartcl}

\usepackage{braket,amsfonts}

\usepackage{array}

\usepackage[caption=false]{subfig}

\usepackage{tikz, pgfplots}
\pgfplotsset{compat=1.18}

\usetikzlibrary{arrows.meta}
\usetikzlibrary{backgrounds}
\usepgfplotslibrary{patchplots}
\usepgfplotslibrary{fillbetween}
\pgfplotsset{%
    layers/standard/.define layer set={%
        background,axis background,axis grid,axis ticks,axis lines,axis tick labels,pre main,main,axis descriptions,axis foreground%
    }{
        grid style={/pgfplots/on layer=axis grid},%
        tick style={/pgfplots/on layer=axis ticks},%
        axis line style={/pgfplots/on layer=axis lines},%
        label style={/pgfplots/on layer=axis descriptions},%
        legend style={/pgfplots/on layer=axis descriptions},%
        title style={/pgfplots/on layer=axis descriptions},%
        colorbar style={/pgfplots/on layer=axis descriptions},%
        ticklabel style={/pgfplots/on layer=axis tick labels},%
        axis background@ style={/pgfplots/on layer=axis background},%
        3d box foreground style={/pgfplots/on layer=axis foreground},%
    },
}

\usepackage{amsmath}
\usepackage{amssymb}
\usepackage{xcolor}

\usepackage[breaklinks=true]{hyperref}

\colorlet{RefColor}{green!50!black}
\colorlet{LinkColor}{red!50!black}
\hypersetup{
  colorlinks = true,
  linkbordercolor = {white},
  linkcolor = LinkColor,
  anchorcolor=LinkColor,
  citecolor= RefColor,
  filecolor=cyan,
  menucolor=red,
  runcolor=cyan,
  urlcolor = LinkColor,
}

\usepackage{algorithm}
\usepackage{algpseudocode}

\usepackage{graphicx,epstopdf}
\DeclareGraphicsExtensions{.pdf,.eps,.png,.jpg,.jpeg}

\usepackage{amsopn}
\usepackage{amsthm}
\usepackage{amsmath}

\newtheorem{proposition}{Proposition}
\newtheorem{remark}{Remark}
\usepackage{xspace}
\usepackage{bold-extra}
\usepackage[most]{tcolorbox}
\usepackage[font={small,it}]{caption}

\colorlet{texcscolor}{blue!50!black}
\colorlet{texemcolor}{red!70!black}
\colorlet{texpreamble}{red!70!black}
\colorlet{codebackground}{black!25!white!25}

\DeclareOldFontCommand{\sc}{\normalfont\scshape}{\@nomath\sc}

\colorlet{header1}{blue!10!black}

\newcommand{\email}[1]{\protect\href{mailto:#1}{#1}}
\newcommand\keywordsname{Key words}
\newcommand\keywordname{Key word}
\newcommand\MSCcodesname{MSC codes}
\newcommand\MSCcodename{MSC code}
\newenvironment{@abssec}[1]{%
     \if@twocolumn
       \section*{#1}%
     \else
       \vspace{.05in}\footnotesize
       \parindent .2in
         \@hangfrom{\color{header1}\bfseries #1. }\ignorespaces 
     \fi}
     {\if@twocolumn\else\par\vspace{.1in}\fi}
\newenvironment{keywords}{\begin{@abssec}{\keywordsname}}{\end{@abssec}}

\newenvironment{MSCcodes}{\begin{@abssec}{\MSCcodesname}}{\end{@abssec}}

\lstdefinestyle{siamlatex}{%
  style=tcblatex,
  texcsstyle=*\color{texcscolor},
  texcsstyle=[2]\color{texemcolor},
  keywordstyle=[2]\color{texemcolor},
  moretexcs={cref,Cref,maketitle,mathcal,text,headers,email,url},
}

\tcbset{%
  colframe=black!75!white!75,
  coltitle=white,
  colback=codebackground, %
  colbacklower=white, %
  fonttitle=\bfseries,
  arc=0pt,outer arc=0pt,
  top=1pt,bottom=1pt,left=1mm,right=1mm,middle=1mm,boxsep=1mm,
  leftrule=0.3mm,rightrule=0.3mm,toprule=0.3mm,bottomrule=0.3mm,
  listing options={style=siamlatex}
}

\newtcblisting[use counter=example]{example}[2][]{%
  title={Example~\thetcbcounter: #2},#1}

\newtcbinputlisting[use counter=example]{\examplefile}[3][]{%
  title={Example~\thetcbcounter: #2},listing file={#3},#1}

\DeclareTotalTCBox{\code}{ v O{} }
{ %
  fontupper=\ttfamily\color{black},
  nobeforeafter,
  tcbox raise base,
  colback=codebackground,colframe=white,
  top=0pt,bottom=0pt,left=0mm,right=0mm,
  leftrule=0pt,rightrule=0pt,toprule=0mm,bottomrule=0mm,
  boxsep=0.5mm,
  #2}{#1}

\patchcmd\newpage{\vfil}{}{}{}
\input{Notation.sty}
\renewcommand{\changed}[1]{\textcolor{black}{#1}}

\begin{tcbverbatimwrite}{tmp_\jobname_header.tex}
  \title{\color{header1}{Greedy construction of quadratic manifolds for nonlinear dimensionality reduction and nonlinear model reduction}\thanks{The authors were supported by the US Department of Energy, Office of Scientific Computing Research,  DOE Award DE-SC0019334 (Program Manager Dr.~Steven Lee) and DOE Award DE-SC0024721 (Program Manager Dr.~Margaret Lentz).}}

\author{{Paul Schwerdtner}\thanks{Courant Institute of Mathematical Sciences, New York University (\email{paul.schwerdtner@nyu.edu}, \email{pehersto@cims.nyu.edu}).}
\and {Benjamin Peherstorfer}\footnotemark[2]}
\date{December 2024}

\end{tcbverbatimwrite}
\input{tmp_\jobname_header.tex}

\ifpdf
\hypersetup{ pdftitle={Greedy Quadratic manifolds} }
\fi

\makeatletter
\newcommand{\specialcell}[1]{\ifmeasuring@#1\else\omit$\displaystyle#1$\ignorespaces\fi}
\makeatother

\graphicspath{{PlotSources}}
\begin{document}
\maketitle

\begin{tcbverbatimwrite}{tmp_\jobname_abstract.tex}
  \begin{abstract}
  Dimensionality reduction on quadratic manifolds augments linear approximations with quadratic correction terms. Previous works rely on linear approximations given by projections onto the first few leading principal components of the training data; however, linear approximations in subspaces spanned by the leading principal components alone can miss information that are necessary for the quadratic correction terms to be efficient. In this work, we propose a greedy method that constructs subspaces from leading as well as later principal components so that the corresponding linear approximations can be corrected most efficiently with quadratic terms.
  Properties of the greedily constructed manifolds allow applying linear algebra reformulations so that the greedy method scales to data points with millions of dimensions. Numerical experiments demonstrate that an orders of magnitude higher accuracy is achieved with the greedily constructed quadratic manifolds compared to manifolds that are based on the leading principal components alone.
  \end{abstract}

  \begin{keywords}
    dimensionality reduction, quadratic manifolds, greedy methods, physics applications, model reduction
  \end{keywords}

  \begin{MSCcodes}
  65F55,  %
    62H25, %
    65F30, %
    68T09 %
  \end{MSCcodes}
\end{tcbverbatimwrite}
\input{tmp_\jobname_abstract.tex}
\section{Introduction}
\label{sec:introduction}
Linear dimensionality reduction in subspaces given by the principal component analysis (PCA) can lead to poor approximations when correlations between components of data points are strongly nonlinear \cite{ScholkopfSM1998Nonlinear}. For example, if data points represent transport phenomena, e.g., frames of a video that show a moving coherent structure such as a car or ship driving by, then PCA provides poor dimensionality reduction in the sense that a large number of principal components is necessary to well approximate the data points \cite{10.1145/1970392.1970395,PhysRevFluids.5.054401,Peherstorfer2022Breaking}. To circumvent this limitation of linear dimensionality reduction, nonlinear correction terms can be added to the linear approximations given by the PCA. In this work, we focus on correction terms that are obtained by evaluating a nonlinear feature map at the linear approximations of the data points, which is widely useful in a range of science and engineering applications such as nonlinear model reduction \cite{RozzaHP2008Reduced,BennerGW2015survey,AntoulasBG2020Interpolatory,KramerPW2024Learning} and closure modeling  \cite{Sagaut2006Large,WangABI2012Proper,DuraisamyIX2019Turbulence,CoupletSB2003Intermodal,XieMRI2018Data-Driven,GouasmiPD2017priori,ZannaB2020Data-Driven,PanD2018Data-Driven}.

The motivation for us are the works \cite{GeelenWW2023Operator,BarnettF2022Quadratic} that propose to find a matrix to weight the output of a polynomial feature map so that the corresponding correction term efficiently reduces the error of PCA approximations. The authors of  \cite{GeelenWW2023Operator,BarnettF2022Quadratic} show that such corrections are helpful precisely for data that represent transport phenomena. Rather than using the subspace given by the PCA for the linear approximations as in \cite{GeelenWW2023Operator,BarnettF2022Quadratic}, we construct the subspace together with the weight matrix for the given feature map. The goal is to construct a subspace for the linear approximations that leads to the lowest error with the correction terms based on the given feature map, which is not necessarily obtained with the PCA subspaces. Recall that the feature map is evaluated at the linear approximations of the data points (rather than the original, high-dimensional data points) and thus the linear subspace approximations need to carry the information that are necessary for the feature map to provide efficient corrections.
We propose a greedy method that selects the basis vectors of the subspace of the linear approximations based on the given feature map and show with various examples that greedily selecting the subspace outperforms by orders of magnitude the accuracy obtained when correcting PCA approximations.

There is a wide range of nonlinear dimensionality reduction methods \cite{TenenbaumSL2000Global,RoweisS2000Nonlinear,DonohoG2003Hessian,BelkinN2003Laplacian,MaatenH2008Visualizing,HintonS2006Reducing}. 
However, we focus specifically on nonlinear approximations that are obtained by correcting linear approximations with nonlinear terms given by feature maps that are evaluated at the linear approximations, which is useful in nonlinear model reduction and closure modeling, as mentioned above. We further focus on polynomial feature maps because polynomial nonlinear terms are pervasive in applications: Polynomial manifolds have been used in \cite{QiaoZWZ2013explicit} and quadratic manifolds in \cite{JainTRR2017quadratic,RutzmoserRTJ2017Generalization}. Quadratic manifolds are used to model latent dynamics in \cite{GeelenBW2023Learning, BarnettF2022Quadratic,goyal2024generalized,https://doi.org/10.1002/pamm.202200049,SHARMA2023116402,yildiz2024datadriven} and have been shown to achieve higher accuracy than linear approximations as given by, e.g., dynamic mode decomposition and related linear methods \cite{rowley2009spectral,schmid2010dynamic,tu2013dynamic,kutz2016dynamic}. Quadratic polynomials are important also for formulating nonlinear system dynamics with guaranteed stability as in \cite{PhysRevFluids.6.094401,SAWANT2023115836,goyal2024guaranteed}, which shows that focusing on quadratic manifolds is of relevance in many applications in science and engineering \cite{doi:10.1137/16M1098280,PEHERSTORFER2016196,QIAN2020132401,doi:10.1137/14097255X,5991229,Schlegel_Noack_2015}. 

The works \cite{GeelenWW2023Operator,BarnettF2022Quadratic} fit the weight matrix such that a quadratic feature map is most efficient in a least-squares sense for PCA approximations, which is different from our approach because we greedily select the subspace of the linear approximations rather than using PCA and fitting the weight matrix only. 
The works \cite{GeelenBW2023Learning,GeelenBWW2024Learning} propose an alternating minimization approach to fit the subspace of the linear approximations and the weight matrix of the correction term together; however, the alternating minimization gives approximations without an explicit encoder map. 
Additionally, as we will demonstrate with our numerical experiments, the optimization with alternating minimization can be computationally demanding and is orders of magnitude slower than the proposed greedy approach. In fact, in our numerical examples with high-dimensional data vectors, the alternating minimization approach became intractable in terms of runtime, which is also due to its slow convergence \cite{ChenHYY2014direct}. \changed{We remark that greedy methods have been developed for constructing linear approximation spaces in model reduction \cite{10.1115/1.1448332,https://doi.org/10.1002/fld.867,doi:10.1137/070694855,GreplGreedy}.}

This manuscript is organized as follows. We first provide preliminaries in Section~\ref{sec:preliminaries} and state the problem formulation. The greedy construction of quadratic manifolds is introduced in Section~\ref{sec:our_method} \changed{and its properties are discussed in Section~\ref{sec:Disc}}. Numerical experiments are shown in Section~\ref{sec:numerical_experiments} and conclusions are drawn in Section~\ref{sec:Conc}.

\section{Preliminaries}
\label{sec:preliminaries}
In this section, we briefly discuss dimensionality reduction on manifolds based on feature maps and introduce the problem formulation.

\subsection{Linear approximations in subspaces} %
\label{sub:dimensionality_reduction}
Let $\Vcal \subset \mathbb{R}^{\nfull}$ be an $r$-dimensional subspace of $\mathbb{R}^{\nfull}$. Let further $\Vlin = [\bfv_1, \dots, \bfv_r] \in \mathbb{R}^{\nfull \times r}$ be a basis matrix of $\Vcal$ that has orthonormal columns with respect to the Euclidean inner product. 
We denote the orthogonal projection operator corresponding to $\Vcal$  and the Euclidean inner product as $\bfP_{\Vcal}: \mathbb{R}^{\nfull} \to \Vcal$, which we interpret as a matrix $\bfP_{\Vcal} = \Vlin\Vlin^{\top}\in \mathbb{R}^{\nfull \times \nfull}$. 
The projection onto $\Vcal$ can be written as the composition of an encoder $f_{\Vlin}: \mathbb{R}^{\nfull} \to \mathbb{R}^{\nred}, \fullstate \mapsto \Vlin^{\top}\fullstate$ and a decoder $g_{\Vlin}: \mathbb{R}^{\nred} \to \mathbb{R}^{\nfull}, \redstate \mapsto \Vlin\redstate$. Notice that $f_{\Vlin}$ and $g_{\Vlin}$ are linear. Notice further that an encoder is sometimes called embedding map and a decoder a lifting map.

Consider now $\nsnapshots \in \mathbb{N}$ data points $\fullstatei{1}, \dots, \fullstatei{\nsnapshots} \in \R^{\nfull}$ of dimension $\nfull$, which we collect as columns in a data matrix $\snapshots = [\fullstatei{1}, \dots, \fullstatei{k}] \in \R^{\nfull \times \nsnapshots}$. We refer to $f_{\Vlin}(\fullstate) = \Vlin^{\top}\fullstate = \redstate$ as the encoded data point $\redstate \in \mathbb{R}^{\nred}$ and to $g_{\Vlin}(f_{\Vlin}(\fullstate)) = \Vlin\redstate$ as the approximated data point. The sum of the errors of projecting the data points onto $\Vcal$ is
\begin{equation}\label{eq:Prelim:ProjError}
  \reconstructionError(\Vcal, \snapshots) = \|\snapshots - \bfP_{\Vcal}\snapshots\|_F\,.
\end{equation}
Recall that the lowest projection error \eqref{eq:Prelim:ProjError} for an $\nred$-dimensional subspace of the column space of $\snapshots$ is obtained by the PCA space $\Scal \subset \mathbb{R}^{\nfull}$ spanned 
by the left-singular vectors $\USVDi{1}, \dots, \USVDi{\nred} \in \mathbb{R}^{\nfull}$ of $\snapshots$ with the $r \leq \nsnapshots$ largest singular values $\sigma_1 \geq \dots \geq \sigma_r \geq \sigma_{r + 1} \geq \dots \geq \sigma_{\nsnapshots}$.

\subsection{Manifold approximations via nonlinear decoders}
Adding nonlinear corrections can lead to approximations with lower errors than PCA. %

\subsubsection{Corrections that depend on encoded data points only}
Consider a decoder with a nonlinear correction as
\begin{equation}\label{eq:Prelim:LiftingWithH}
  g_{\Vlin,H}(\redstate) = \Vlin\redstate + H(\redstate)\,,
\end{equation}
where $H: \mathbb{R}^{\nred} \to \mathbb{R}^{\nfull}$ is a nonlinear function. 
Using the encoder $f_{\Vlin}(\fullstate) = \Vlin^{\top}\fullstate$, a data point $\fullstate \in \mathbb{R}^{\nfull}$ can be approximated as
\[
  (g_{\Vlin, H} \circ f_{\Vlin})(\fullstate) =  \bfP_{\Vcal}\fullstate + H(f_{\Vlin}(\fullstate))\,,
\]
which shows that the correction term $H(f_{\Vlin}(\fullstate)) = H(\Vlin^{\top}\fullstate) \in \mathbb{R}^{\nfull}$ is added to the best approximation in $\Vcal$ given by the orthogonal projection $\bfP_{\Vcal}\fullstate$. 

Because the nonlinear correction term is added when decoding (``lifting'') back to the high-dimensional representation, the correction term $H(\Vlin^{\top}\fullstate)$ depends nonlinearly on the encoded data point $f_{\Vlin}(\fullstate) = \Vlin^{\top}\fullstate$ %
only, rather than on the original, high-dimensional data point $\fullstate$. Thus, the orthogonal parts of $\fullstate$ with respect to $\Vcal$ cannot inform the additive correction.  We can equivalently say that the correction term depends only on the projected data point $\bfP_{\Vcal}\fullstate = \Vlin\Vlin^{\top}\fullstate$ because the linear decoding $\Vlin\redstate$ is just a different representation of the encoded data point $\Vlin^{\top}\fullstate$. That the correction depends only on the projected data point $\bfP_{\Vcal}\fullstate$ is important in many applications where only the encoded point $\redstate$ is available such as in model reduction \cite{RozzaHP2008Reduced,BennerGW2015survey,AntoulasBG2020Interpolatory,KramerPW2024Learning} and closure modeling \cite{Sagaut2006Large,WangABI2012Proper,DuraisamyIX2019Turbulence,CoupletSB2003Intermodal,XieMRI2018Data-Driven,GouasmiPD2017priori,ZannaB2020Data-Driven,PanD2018Data-Driven,Uy2021}. 

\subsubsection{Manifold approximations given by nonlinear corrections}
Given a subspace $\Vcal$ and a map $H$ that induces a correction term, the decoder $g_{\Vlin, H}$ and encoder $f_{\Vlin}$ lead to the manifold
\begin{equation}\label{eq:Prelim:ManifoldDef}
  \manifold_{\nred}(\Vlin, H)=\{ g_{\Vlin, H}(\redstate) \,|\, \redstate \in \R^{\nred} \} \subset \R^{\nfull}\,.
\end{equation}
Because $H$ is nonlinear, the manifold $\manifold_{\nred}$ can contain points in $\mathbb{R}^{\nfull}$ that are outside of the subspace $\Vcal$. 
Notice that the image of a linear decoder map $g_{\Vlin}$ is the $\nred$-dimensional (linear) subspace $\Vcal$ of $\R^{\nfull}$ that is spanned by the columns of $\Vlin$. 

\subsubsection{Correction terms via polynomial feature maps}
\label{subsubsec:prev_polynomial_feature_map}
The works~\cite{JainTRR2017quadratic,RutzmoserRTJ2017Generalization,GeelenWW2023Operator,BarnettF2022Quadratic} allow correction maps $H$ that are of the form
\[
  H(\redstate) = \Vnonlin h(\redstate)\,,
\]
where $\Vnonlin \in \mathbb{R}^{\nfull \times \nredmod}$ is a matrix and $h: \mathbb{R}^{\nred} \to \mathbb{R}^{\nredmod}$ is a feature map that lifts the encoded data point $\redstate$ onto a $\nredmod$-dimensional vector. The matrix $\Vnonlin$ can be understood as a weight matrix to obtain an $\nfull$-dimensional correction from the $\nredmod$-dimensional feature vector $h(\redstate)$. 
We denote the corresponding decoder as 
\begin{equation}\label{eq:GVW}
  g_{\Vlin,\Vnonlin}(\redstate) = \Vlin\redstate + \Vnonlin h(\redstate)\,, 
\end{equation}
which highlights that the feature map $h$ is given and only $\Vlin$ and $\Vnonlin$ can be fitted to data.

In \cite{JainTRR2017quadratic,RutzmoserRTJ2017Generalization,GeelenWW2023Operator,BarnettF2022Quadratic}, the feature map $h$ is a polynomial function such as a quadratic
\begin{equation}
  \label{eq:condensed_kronecker}
  \kronfeaturemap: \R^{\nred} \to \R^{\nred(\nred+1)/2},  \redvec \mapsto \begin{bmatrix}
    \redveci{1} \redveci{1} & \redveci{1} \redveci{2} & \dots & \redveci{1} \redveci{\nred} & \redveci{2} \redveci{2} &\dots& \redveci{r} \redveci{r}
  \end{bmatrix}^\top\,,
\end{equation}
in which case we refer to the manifold $\manifold_{\nred}$ defined in \eqref{eq:Prelim:ManifoldDef} as quadratic manifold. Feature maps with higher order polynomials are considered in, e.g.,  \cite{GeelenBWW2024Learning,GeelenBW2023Learning}.

Given a data matrix $\snapshots$ and a feature map $h$, the authors of \cite{GeelenWW2023Operator,BarnettF2022Quadratic} set the columns of $\Vlin$ to be the leading $\nred$ left-singular vectors of $\snapshots$ and then fit $\Vnonlin$ via a regularized linear least-squares problem to minimize the error of approximating $\snapshots$ as
\begin{equation}\label{eq:Prelim:LinearLSQApproach}
  \min_{\Vnonlin \in \mathbb{R}^{\nfull \times \nredmod}} \|\bfP_{\Vcal}\snapshots + \Vnonlin h(f_{\Vlin}(\snapshots)) - \snapshots\|_F^2 + \gamma \|\Vnonlin\|_F^2 \,,
\end{equation}
where we overload the notation of $h$ to allow $h$ to be evaluated column-wise on the matrix $f_{\Vlin}(\snapshots) = \Vlin^{\top}\snapshots$ to obtain $h(f_{\Vlin}(\snapshots)) \in \mathbb{R}^{\nredmod \times \nsnapshots}$.
The regularization term is controlled by $\gamma > 0$ and can prevent overfitting of the weight matrix $\Vnonlin$ to data in $\snapshots$.

\begin{remark}
  In~\cite{GeelenBW2023Learning, GeelenBWW2024Learning}, the authors propose an alternating minimization approach to fit $\Vlin$ and $\Vnonlin$; however, in doing so, the authors also obtain a nonlinear encoder $f$ and thus the approximations obtained in \cite{GeelenBW2023Learning, GeelenBWW2024Learning} are not of the type that can be described with a nonlinear decoder $g_{\Vlin,\Vnonlin}$ as \eqref{eq:GVW} and a linear encoder $f_{\Vlin}$ with the same $\Vlin$ as used in the decoder $g_{\Vlin,\Vnonlin}$. \changed{In fact, the encoder proposed in~\cite{GeelenBW2023Learning, GeelenBWW2024Learning} consists of a nonlinear optimization problem that aims to minimize the reconstruction error and the encoder function is thus} not available in closed form. Details about the alternating minimization approach are given in Appendix~\ref{appx:AlternatingMin}. Even though the alternating minimization approach leads to a different setting, we will numerically compare to it later. \changed{Moreover, we will compare the performance of linear and nonlinear encoders to embed points onto the same quadratic manifold.}
  The work \cite{BarnettFM2023Neural-network-augmented} also uses the leading $\nred$ left-singular vectors to span $\Vcal$ but then parametrizes the map $H_{\bftheta}: \mathbb{R}^{\nred} \to \mathbb{R}^{\nfull}$ with a neural network. Thus, instead of having given a feature map $h$ and fitting only the weight matrix $\Vnonlin$ for a decoder of the form given in  \eqref{eq:GVW}, the authors of \cite{BarnettFM2023Neural-network-augmented} fit the parameter vector $\bftheta$ of the neural network $H_{\bftheta}$ to minimize the error of approximating the left-singular vectors of index greater than $\nred$.
\end{remark}

\begin{figure}
  \begin{center}
    \resizebox{0.99\columnwidth}{!}{\scriptsize\input{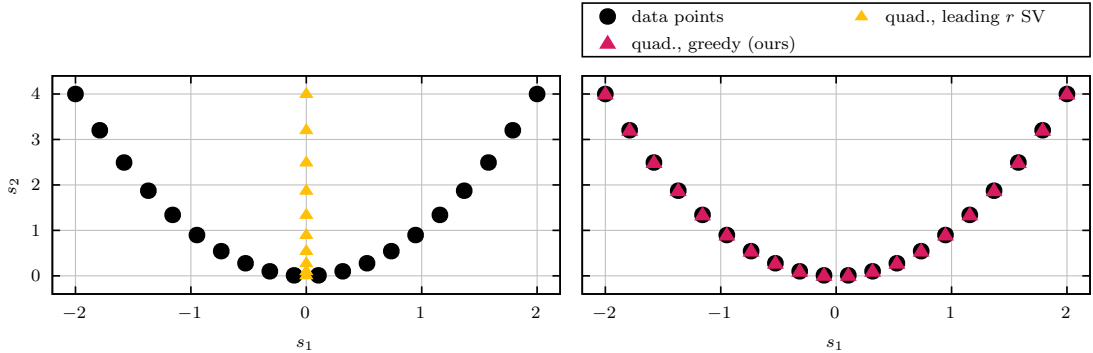}}
    \vspace{-0.5cm}
  \end{center}
  \caption{The plots show the data points given by the data matrix $\snapshots^{(\mathrm{parabola})}$ and their respective approximation on one-dimensional quadratic manifolds. Constructing the quadratic manifold based on the leading $\nred = 1$ left-singular vector alone leads to poor approximations, as can be seen in the left plot. In contrast, greedily selecting the subspace $\Vcal$ with the proposed approach leads to quadratic-manifold approximations that exactly represent the data points, see plot on the right.} 
  \label{fig:quadraticillustration}
\end{figure}

\subsection{Problem formulation}
\label{sec:Prelim:problem_formulation}
We now illustrate on a toy example that letting $\Vcal$ be spanned by the leading $\nred$ left-singular vectors of the data matrix $\snapshots$ can lead to inefficient corrections and thus poor approximations.

Recall the linear encoder $f_{\Vlin}$ that depends on $\Vlin$ and the nonlinear decoder $g_{\Vlin,\Vnonlin}$ defined in \eqref{eq:GVW} that depends on $\Vlin$ and $\Vnonlin$ and additionally on a given feature map $h$. Consider the data matrix $\snapshots^{(\mathrm{parabola})}$ with columns 
\begin{equation}\label{eq:Prelim:ParabolaDataPoints}
  \fullstatei{i} = \begin{bmatrix}-2+4(i-1)/19\\ (-2+4(i-1)/19)^2\end{bmatrix} \in \mathbb{R}^2\,,\qquad i = 1, \dots, 20\,,
\end{equation}
which are plotted in Figure~\ref{fig:quadraticillustration}a. 
With 
\begin{equation}\label{eq:Prelim:Exact}
  \Vlin = [1, 0]^{\top} \in \mathbb{R}^{2 \times 1}\,,\qquad \Vnonlin = [0, 1]^{\top} \in \mathbb{R}^{2 \times 1}\,,
\end{equation}
the data points \eqref{eq:Prelim:ParabolaDataPoints} can be exactly represented in dimension $\nred = 1$ with the decoder $g_{\Vlin,\Vnonlin}$, linear encoder $f_{\Vlin}$, and quadratic feature map $\kronfeaturemap$ defined in \eqref{eq:condensed_kronecker}. 
However, notice that the subspace spanned by the columns of $\Vlin$ is not the subspace spanned by the first left-singular vector of $\snapshots^{(\mathrm{parabola})}$, which is $\Scal = \operatorname{span}\{[0, 1]^T\}$. If $\Vcal = \Scal$, then the encoder ignores the first component of all data points \eqref{eq:Prelim:ParabolaDataPoints}, which means that the reduced data points 
\[
  \redstatei{i} = [(-2+4(i-1)/19)^2]\,,\qquad i = 1, \dots, 20\,,
\]
are not informative for determining the corrections with the feature map $h$, independent of the weight matrix $\Vnonlin$. Recall that $h$ is evaluated at the encoded data points $\redstate$ rather than the original, high-dimensional data points $\fullstate$. Thus, in this toy example, the information carried by the first component of the data points is lost in the encoded data points when projecting onto the first $\nred = 1$ leading left-singular vectors and thus the encoded data points are un-informative for finding a correction with $h$.

The observation that the choice of the subspace $\Vcal$ is critical for the quality of the approximations with a given feature map can also be explained with the insights given in~\cite{CohenFSM2023Nonlinear}, where it is noted that the encoder \eqref{eq:GVW} can be interpreted as taking nonlinear measurements $\Vnonlin h(\redstate)$ of the projected data point $\Vlin\redstate = \bfP_{\Vcal}\fullstate$. If $\Vcal$ and the feature map $h$ are incompatible in the sense that the subspace $\Vcal$ does not carry information needed for the feature map $h$ to provide informative measurements, then the correction cannot be efficient.

\section{Greedy construction of quadratic manifolds}
\label{sec:our_method}
We propose a method to construct subspaces $\Vcal$ specifically for a given feature map $h$ such that the data $\snapshots$ is well approximated on the corresponding manifold $\manifold_{\nred}$, instead of using a subspace that is agnostic to the feature map $h$ such as the space spanned by the first $\nred$ left-singular vectors of $\snapshots$. 
We introduce a greedy method that selects a set of basis vectors of $\Vcal$ from the first $m \gg \nred$ left-singular vectors of $\snapshots$, instead of simply taking the first $\nred$ only. 

In Section~\ref{sec:Greedy:Method}, we introduce the greedy method to construct quadratic manifolds as well as other manifolds for given feature maps $h$. We further introduce in Section~\ref{sec:Greedy:CompRed} a computational procedure for the greedy method that builds on linear least-squares problems with unknowns scaling independently of the dimension $\nfull$ and instead with the number of data points $\nsnapshots$, which typically is smaller than $\nfull$. Section~\ref{sec:Greedy:AlgorithmicDescription} provides an algorithmic description.

\subsection{Greedy selection strategy}\label{sec:Greedy:Method}
Recall that $\USVDi{1}, \dots, \USVDi{\nsnapshots} \in \mathbb{R}^{\nfull}$ are the left-singular vectors of the data matrix $\snapshots$ ordered descending with respect to the singular values $\sigma_1 \geq \dots \geq \sigma_{\nsnapshots}$. We now greedily select $\nred$ left-singular vectors $\USVDi{j_1}, \dots, \USVDi{j_{\nred}}$ with indices $j_1, j_2, \dots, j_{\nred} \in \mathbb{N}$ from the first $\nconsider \gg \nred$ left-singular vectors $\USVDi{1}, \dots, \USVDi{\nconsider}$. We stress that the indices $j_1, \dots, j_{\nred}$ of the left-singular vectors that we select do not necessarily correspond to the first $\nred$ left-singular vectors with the largest singular values. 

Let $i = 1, \dots, \nred$ be the iteration counter variable of the greedy selection and define $\Vcal_i$ as the subspace at iteration $i$ that is spanned by the columns of the basis matrix $\Vlin_i = [\USVDi{j_1}, \dots, \USVDi{j_i}]$. Analogously to \eqref{eq:Prelim:LinearLSQApproach}, we define the objective function
\begin{equation}\label{eq:Greedy:J}
J(\bfv, \Vlin, \Vnonlin) = \|\bfP_{\Vcal \oplus \operatorname{span}\{\bfv\}}\snapshots + \Vnonlin h(f_{[\Vlin, \bfv]}(\snapshots)) - \snapshots\|_F^2 + \gamma \|\Vnonlin\|_F^2\,,
\end{equation}
where the subspace $\bfP_{\Vcal \oplus \operatorname{span}\{\bfv\}}$ denotes the orthogonal projection operator of the subspace $\Vcal \oplus \operatorname{span}(\bfv)$ spanned by the columns of $\Vlin$ and the vector $\bfv$. 
The function $f_{[\Vlin, \bfv]}$ is the linear encoder corresponding to the space $\Vcal \oplus \operatorname{span}\{\bfv\}$ with basis matrix $[\Vlin, \bfv]$. 
At iteration $i$, we select the left-singular vector $\USVDi{j_i}$ with index $j_i$ that minimizes the objective $J$ defined in \eqref{eq:Greedy:J} over all $\Vnonlin \in \mathbb{R}^{\nfull \times \nredmod}$ and the subspace $\Vcal_{i - 1}$ of the previous iteration $i - 1$,
\begin{equation}\label{eq:Greedy:GreedyOptiProblem}
\min_{j_i = 1, \dots, m} \min_{\Vnonlin \in \mathbb{R}^{\nfull \times \nredmod}} J(\USVDi{j_i}, \Vlin_{i-1}, \Vnonlin)\,,
\end{equation}
where we start  at iteration $i = 0$ with the subspace $\Vcal_0$ that contains only the zero element.
After $\nred$ iterations, we obtain the basis matrix $\Vlin = [\USVDi{j_1}, \dots, \USVDi{j_{\nred}}] \in \mathbb{R}^{\nfull \times \nred}$ and compute the corresponding weight matrix $\Vnonlin$ by solving \eqref{eq:Prelim:LinearLSQApproach}, which give rise to the nonlinear decoder $g_{\Vlin,\Vnonlin}$ defined in \eqref{eq:GVW} with feature map $h$ and the linear encoder $f_{\Vlin}$.

\subsection{Accelerating repeated least-squares solves for efficient greedy selection}\label{sec:Greedy:CompRed}
We now discuss how to re-use a pre-computed singular value decomposition (SVD) of the data matrix $\snapshots$ to accelerate the repeated solves of the least-squares problem for $\Vnonlin$ in \eqref{eq:Greedy:GreedyOptiProblem}. 

\subsubsection{Re-using the pre-computed SVD of the data matrix}
In each greedy iteration $i = 1, \dots, \nred$, the optimization problem \eqref{eq:Greedy:GreedyOptiProblem} is solved. We now show how we can re-use the SVD of the data matrix $\snapshots$, which has to be computed to obtain the principal components, to also reduce the costs of the inner least-squares problem over $\Vnonlin$ in \eqref{eq:Greedy:GreedyOptiProblem}.

Let $\snapshots = \USVD \SigmaSVD \VSVD^\top$ be the SVD of the data matrix $\snapshots$, where the left-singular vectors are the columns of $\USVD$, the right-singular vectors are the columns of $\VSVD$, and the singular values are in descending order on the diagonal of $\SigmaSVD$. Recall that at iteration $i$ of the greedy procedure, the objective function \eqref{eq:Greedy:J} is minimized for $\Vnonlin$ at the subspace $\Vcal_{i - 1}$ spanned by the columns of the basis matrix $\Vlin_{i - 1} = [\USVDi{j_1}, \dots, \USVDi{j_{i-1}}]$ and over all left-singular vectors $\USVDi{1}, \dots, \USVDi{\nconsider}$ up to index $\nconsider$. In particular, for all $j^{\prime} = 1, \dots, \nconsider$, the objective \eqref{eq:Greedy:J} depends on the projection error $\bfP_{\Vcal_{i - 1} \oplus \operatorname{span}\{\USVDi{j^{\prime}}\}}\snapshots - \snapshots$. Because $\Vcal_{i-1}$ is spanned by left-singular vectors of $\snapshots$ and $\USVDi{j^{\prime}}$ is also a left-singular vector, the projection error can be represented as
\begin{equation}\label{eq:Greedy:ProjErrorRepresentation}
\bfP_{\Vcal_{i - 1} \oplus \operatorname{span}\{\USVDi{j^{\prime}}\}}\snapshots - \snapshots = \USVD_{\indexout{i-1}\setminus\{j^{\prime}\}}\SigmaSVD_{\indexout{i-1}\setminus\{j^{\prime}\}}\VSVD_{\indexout{i-1}\setminus\{j^{\prime}\}}^{\top}\,,
\end{equation}
where $\indexin{i-1} = \{j_1, \dots, j_{i - 1}\}$ and $\indexout{i-1} = \{1, \dots, \nsnapshots\}\setminus\indexin{i-1}$ is the complement set of $\indexin{i-1}$. The matrix $\USVD_{\indexout{i-1}\setminus\{j^{\prime}\}}$ contains as columns all left-singular vectors with indices in $\indexout{i-1}\setminus\{j^{\prime}\}$, which are the indices $1, \dots, \nsnapshots$ except $j_1, \dots, j_{i - 1}$ and $j^{\prime}$. Analogously, $\VSVD_{\indexout{i-1}\setminus\{j^{\prime}\}}$ and $\SigmaSVD_{\indexout{i-1}\setminus\{j^{\prime}\}}$ contain the right-singular vectors and the singular values, respectively, corresponding to the indices in $\indexout{i-1}\setminus\{j^{\prime}\}$. For computing the term $h(f_{[\Vlin,\bfv]}(\snapshots))$ in \eqref{eq:Greedy:J}, we can analogously use
\[
[\Vlin_{i-1},  \USVDi{j^{\prime}}]^{\top}\snapshots = \SigmaSVD_{\indexout{i-1}\setminus\{j^{\prime}\}}\VSVD_{\indexout{i-1}\setminus\{j^{\prime}\}}^{\top}\,.
\]

We summarize that it is sufficient to compute the SVD of $\snapshots$ once and then to re-use it to evaluate the objectives during the greedy iterations without having to compute SVDs of the intermediate matrices containing data points again.

\subsubsection{Reduced number of unknowns in least-squares problems}\label{sec:Greedy:RuntimeImprovement}
At iteration $i$ of the greedy procedure, minimizing the objective function over $\Vnonlin \in \mathbb{R}^{\nfull \times \nredmod}$ can be interpreted as solving $l = 1, \dots, \nfull$ linear least-squares problems with the $\nsnapshots \times \nredmod$ system matrix $h(f_{[\Vlin, \bfv]}(\snapshots))^{\top}$ and right-hand sides given by the columns of $(\bfP_{\Vcal \oplus \operatorname{span}\{\bfv\}}\snapshots - \snapshots)^{\top} \in \mathbb{R}^{\nsnapshots \times \nfull}$.
Each least-squares problem provides one of the $\nfull$ rows of $\Vnonlin$. In many cases of interest, the number of data points $\nsnapshots$ is smaller than the dimension of the data points $\nfull$; see numerical examples in Section~\ref{sec:numerical_experiments}.%

Recall that with \eqref{eq:Greedy:ProjErrorRepresentation}, we have given the SVD of the projection error, which is used in the objective \eqref{eq:Greedy:J}.
\changed{Noting that $\|\ALemma\CLemma\|_F=\|\ALemma\|_F$ for any $\ALemma \in \R^{\nLemma \times \pLemma}$ and $\CLemma \in \R^{\pLemma \times \qLemma}$, where $\CLemma$ has orthonormal rows and $\qLemma \geq \pLemma$}
and realizing that $\USVD_{\indexout{i-1}\setminus\{j^{\prime}\}}$ is a matrix with orthonormal rows and as long as $\nsnapshots \leq \nfull$, we obtain that evaluating the objective function $J$ defined in \eqref{eq:Greedy:J} at the minimum $\Vnonlin \in \mathbb{R}^{\nfull \times \nredmod}$ gives the same objective value as evaluating the objective function
\begin{align}
\label{eq:Jprime}
J^{\prime}(\USVDi{j^{\prime}}, \Vlin_{i - 1}, \Vnonlin^{\prime}) = \left\|\SigmaSVD_{\indexout{i-1}\setminus\{j^{\prime}\}}\VSVD_{\indexout{i-1}\setminus\{j^{\prime}\}}^{\top} + \Vnonlin^{\prime} h(f_{[\Vlin_{i-1}, \USVDi{j^{\prime}}]}(\snapshots)) \right\|_F^2 + \gamma \|\Vnonlin^{\prime}\|_F^2
\end{align}
at its respective minimum $\Vnonlin^{\prime} \in \mathbb{R}^{(\nsnapshots -i) \times \nredmod}$. 
The objective $J^{\prime}$ describes $\nsnapshots - i$ many linear least-squares problems, rather than $\nfull$ many as the formulation via the objective $J$ given in \eqref{eq:Greedy:J}. In particular, the dimension of the unknown $\Vnonlin^{\prime}$ is independent of the dimension of the data points $\nfull$. Notice that the objective $J^{\prime}$ is only minimized to compute the minimal objective value of $J$ and that $\Vnonlin^{\prime}$ and $\Vnonlin$ are not used in intermediate greedy iterations to solve \eqref{eq:Greedy:GreedyOptiProblem}. 

\subsection{Algorithm description}\label{sec:Greedy:AlgorithmicDescription}
The proposed greedy method is  described in  Algorithm~\ref{alg:ourquadmani}. The algorithm takes as input the data matrix $\snapshots$, the reduced dimension $\nred$, a regularization parameter $\gamma$, the feature map $\featuremap$ and the number $\nconsider$ of candidate singular vectors to consider. The algorithm first computes the SVD of the data matrix $\snapshots$, which is then re-used to rapidly evaluate the objective $J$ defined in \eqref{eq:Greedy:J} at the minimum via the objective $J^{\prime}$ defined in \eqref{eq:Jprime}. The algorithm therefore iterates over the dimensions $i = 1, \dots, \nred$ and solves in each iteration problem \eqref{eq:Greedy:GreedyOptiProblem} to obtain the index $j_i$ of the left-singular vector to expand the subspace $\Vcal_{i - 1}$ of the previous iteration. Notice that the algorithm uses the objective $J^{\prime}$ given in \eqref{eq:Jprime}. 
The sets $\indexin{i}$ and $\indexout{i}$ are then updated and used in the next iteration to compute the intermediate objective \eqref{eq:Jprime}.
The algorithm terminates when the subspace $\Vcal$ of dimension $\nred$ has been constructed.

\begin{algorithm}[t]
  \caption{Greedy construction of quadratic manifolds}
  \label{alg:ourquadmani}
  \begin{algorithmic}[1]
    \Procedure{GreedyQM}{$\snapshots, \nred, \gamma, \featuremap, \nconsider$}
    \State{Compute the SVD of the snapshot matrix $\USVD \SigmaSVD \VSVD^\top = \snapshots$} 
    \State{Set $\indexin{0}=\{\}, \indexout{0}=\{1, \dots, \nsnapshots\}, \Vlin_0 = []$}
    \For{$i = 1,\dots, r$}
    \State{Compute $\USVDi{j_i}$ that minimizes~\eqref{eq:Jprime} over all $\USVDi{j^{1}}, \dots, \USVDi{j^m}$ and $\Vnonlin^{\prime} \in \mathbb{R}^{(\nsnapshots -i) \times \nredmod}$}
    \State{Set $\indexin{i}=\{j_1, \dots, j_i\}$ and $\indexout{i}=\{1, \dots, \nsnapshots\}\setminus \indexin{i}$}
    \State{Set $\Vlin_i = [\USVDi{1}, \dots, \USVDi{j_i}]$}
    \EndFor
    \State{Set $\Vlin=[\USVDi{j_1}, \dots, \USVDi{j_r}]$}
    \State{Compute $\Vnonlin$ via the regularized least-squares problem \eqref{eq:Prelim:LinearLSQApproach}.}
    \State{Return $\Vlin$ and $\Vnonlin$.}
    \EndProcedure
  \end{algorithmic}
\end{algorithm}

{
\color{black}
\section{Discussion}\label{sec:Disc}
In this section, we discuss properties of the proposed greedy method and provide insights about its performance. 

\subsection{Bounding error of quadratic-manifold approximations fro below}
\label{sec:lower_bounds}
Lower bounds of the error of approximating data points on quadratic manifolds are studied in~\cite{BuchfinkGH2023Approximation}. However, the lower bounds in \cite{BuchfinkGH2023Approximation} depend on bounds of the Kolmogorov n-width of the set of elements that is to be approximated on the quadratic manifold, which are unavailable in our situation. In contrast, in this section, we build on~\cite{BuchfinkGH2023Approximation} to derive lower bounds that hold for given data sets. 
The lower bound that we show in the following is analogous to the results for linear approximation spaces, for which the projection error in the Frobenius norm can be bounded with sums of the truncated squared singular values of the corresponding data matrix \cite{GolubV-L2013Matrix}. We now prove that for a given data matrix $\snapshots$, the error of approximations on quadratic manifolds of dimension $\nred$ cannot be lower than the error of approximations in linear approximation spaces of dimension $\nredmod+\nred$, which in turn is lower bounded by the truncated sum of the squared singular values. It is important to note that the following lower bound holds independent of the encoder that is used.

\begin{proposition}
\label{prop:lower_bound}
Consider a manifold $\mathcal{M}_r(\Vlin, H)$ as defined in \eqref{eq:Prelim:ManifoldDef} with matrix $\Vlin \in \mathbb{R}^{\nfull \times \nred}$ and dimension $\nred \leq \nfull$. The correction map $H: \mathbb{R}^{\nred} \to \mathbb{R}^{\nfull}$ is given via a feature map $\featuremap: \R^{\nred} \to \R^{\nredmod}$ as $H(\redstate)=\Vnonlin \featuremap(\redstate)$, where $p \geq \nred$ and $\Vnonlin \in \mathbb{R}^{\nfull \times \nredmod}$ is a weight matrix. Let now $\snapshots = [\fullstatei{1}, \dots, \fullstatei{k}] \in \R^{\nfull \times  \nsnapshots}$ be a data matrix with singular values $\singval{1},\dots,\singval{\minkn}$, where $\ell = \min(n, k)$ and $\ell \geq \nred$. Then, the averaged error of approximating the columns of $\snapshots$ on $\mathcal{M}_{\nred}(\Vlin, H)$ is bounded from below as 
    \begin{equation}
    \label{eq:LowerBound:prop}
\sum_{i=1}^{\nsnapshots}\min_{\hat{\fullstate}^{(i)}\in \mathcal{M}_r}\|\hat{\fullstate}^{(i)}-\fullstatei{i}\|_2^2 \geq \sum\limits_{i=\nredmod+\nred+1}^{\minkn}\singval{i}^2\,.
\end{equation}
\end{proposition}

\begin{proof} 
    Note that
    \begin{align}
        \sum_{i=1}^{\nsnapshots}\min_{\hat{\fullstate}^{(i)}\in \mathcal{M}_r}\|\hat{\fullstate}^{(i)}-\fullstatei{i}\|_2^2 = \min_{\snapshots_r \in \R^{\nred \times \nsnapshots}}\frobsq{
        \Vlin \snapshots_r + \Vnonlin\featuremap(\snapshots_r)-\snapshots
        }
    \end{align}
    holds. 
    A monotonicity argument yields
    \begin{align*}
        \min_{\snapshots_r \in \R^{\nred \times \nsnapshots}}\frobsq{
        \Vlin \snapshots_r + \Vnonlin\featuremap(\snapshots_r)-\snapshots
        } \ge
        \min_{\snapshots_r \in \R^{\nred \times \nsnapshots}, \widetilde{\snapshots}_r \in \R^{\nredmod \times \nsnapshots}}\frobsq{
        \Vlin \snapshots_r + \Vnonlin\widetilde{\snapshots}_r-\snapshots
        } =
        \sum\limits_{i=\nredmod+\nred+1}^{\minkn}\singval{i}^2,
    \end{align*}
    because we can represent $$\Vlin \snapshots_r + \Vnonlin\widetilde{\snapshots}_r=\begin{bmatrix} \Vlin & \Vnonlin \end{bmatrix} \begin{bmatrix} \snapshots_r^\top & \widetilde{\snapshots}_r^\top \end{bmatrix}^\top$$ and so set it to the best approximation of rank $\nredmod+\nred$ of $\snapshots$, which is given by the truncated singular value decomposition.
\end{proof}
Proposition~\ref{prop:lower_bound} states that the error of approximating the columns of a data matrix $\snapshots$ on a quadratic manifold $\mathcal{M}_r(\Vlin, H)$, i.e., where the feature map underlying $H$ is a quadratic function \eqref{eq:condensed_kronecker}, is lower bounded by the best linear approximation error in an $\nred(\nred + 1)/2 + \nred$-dimensional space.

\subsection{Approximating projections onto later left-singular vectors from encoded data points}
\label{sec:discussion_correlation}
We now discuss requirements for quadratic manifolds to achieve more accurate approximations than linear approximation spaces of the same dimension, which provides further motivation for the proposed greedy method.  
\subsubsection{Projections onto later left-singular vectors}\label{sec:Discussion:Correlation1}
Consider a data matrix $\snapshots \in \mathbb{R}^{\nfull \times \nsnapshots}$. Let now $\Vlin \in \mathbb{R}^{\nfull \times \nred}$ have orthonormal columns that span an $\nred$-dimensional subspace of the column space of $\snapshots$ and let $\breve{\Vlin} \in \mathbb{R}^{\nfull \times (\nsnapshots - \nred)}$ have orthonormal columns that span the orthogonal complement of the column space of $\Vlin$. Let $\fullstate \in \mathbb{R}^{\nfull}$ be a data point that we represent as
\begin{equation}\label{eq:Discussion:StateSVD}
\fullstate = \Vlin\redstate + \breve{\Vlin}\breve{\fullstate}_r\,,
\end{equation}
with $\redstate = \Vlin^{\top}\fullstate$ and $\breve{\fullstate}_r = \breve{\Vlin}^{\top}\fullstate$. 
Let us now compare \eqref{eq:Discussion:StateSVD} with an approximation $\hat{\fullstate} \in \mathcal{M}_r(\Vlin, H)$ that we obtain on a quadratic manifold with basis matrix $\Vlin$ and correction function $H(\redstate) = \Vnonlin h(\redstate)$,
\begin{equation}\label{eq:Discussion:StateQM}
\hat{\fullstate} = \Vlin\redstate + \Vnonlin h(\redstate)\,,
\end{equation}
where we used the linear encoder $f_{\Vlin}(\fullstate) = \Vlin^{\top}\fullstate = \redstate$. 
Comparing the representation of the data point $\fullstate$ given in \eqref{eq:Discussion:StateSVD} with the quadratic-manifold approximation \eqref{eq:Discussion:StateQM}, one can observe that the correction term is supposed to capture $\breve{\Vlin}\breve{\fullstate}_r$. 
If the matrix $\Vlin$ for the encoding $\redstate = f_{\Vlin}(\fullstate) = \Vlin^{\top}\fullstate$ and the weight matrix $\Vnonlin$ of the quadratic manifold are constructed such that \begin{equation}\label{eq:Disc:VSWH}
\breve{\Vlin}\breve{\fullstate}_r = \Vnonlin h(\redstate),
\end{equation}
then the quadratic manifold can exactly represent the data points, $\hat{\fullstate} = \fullstate$. 
There are two key requirements such that $\Vnonlin h(\redstate)$ approximates $\breve{\Vlin}\breve{\fullstate}_r$ well. 
First, the dimension $p$ of the feature space has to be large enough so that the column space of $\Vnonlin$ can be sufficiently high dimensional and rich to well approximate that subspace of the column space of $\breve{\Vlin}$ that is most important for approximating $\fullstate$. 
Second, the encoded data point $\redstate$ lifted into the feature space $h(\redstate)$ has to carry enough information so that it well approximates the coordinates $\Vnonlin^{+}\fullstate$ of the (oblique) projection of $\fullstate$ onto the column space of $\Vnonlin$, where $\Vnonlin^{+}$ denotes the Moore-Penrose pseudo-inverse of $\Vnonlin$. %
Because we fix the encoder to be linear $f_{\Vlin}(\fullstate) = \Vlin^{\top}\fullstate$, this means that we have to choose $\Vlin$ such that the coordinates (encoded data point) $\redstate = \Vlin^{\top}\fullstate$ carry enough information for approximating the coordinates $\Vnonlin^{+}\fullstate$. 
Notice that we briefly discuss these insights already in Section~\ref{sec:Prelim:problem_formulation}. We also stress that the authors of \cite{GeelenBWW2024Learning,CohenFSM2023Nonlinear} and even already \cite{DeanKKO1991Low} provide these insights.

\subsubsection{Greedily establishing correlation between lifted encoded data points and coordinates of projections onto orthogonal complements}
We now discuss why the proposed greedy method finds a basis matrix $\Vlin$ and a weight matrix $\Vnonlin$ such that (a) the lifted encoded data points $h(\Vlin^{\top}\snapshots)$ are informative for obtaining the coordinates of the projection $\Vnonlin^{+}\snapshots$ and (b) the column space of $\Vnonlin$ approximates well the subspace of the orthogonal complement of the column space of $\Vlin$ that is most important for approximating $\snapshots$. 

Recall that our greedy method  %
selects $\nred$ left-singular vectors of $\snapshots$  from the first $\nconsider > \nred$ ones for forming the basis matrix $\Vlin$: %
The index set $\indexin{\nred}=\{j_1, \dots, j_r\}$ includes the indices of the left-singular vectors that form the columns of $\Vlin$. The complement of the set is   $\indexout{\nred}$, which leads to $\breve{\Vlin}$. Let us set $\nconsider = \min(\nfull, \nsnapshots)$ for ease of exposition, then we can write the singular value decomposition of $\snapshots = \leftsings\singvals\rightsings^{\top}$ according to the index sets $\indexin{\nred}$ and $\indexout{\nred}$, 
\begin{align*}
\snapshots=
\begin{bmatrix}
    \leftsings_{\indexin{\nred}} &
    \leftsings_{\indexout{\nred}}
\end{bmatrix}
\begin{bmatrix}
    \singvals_{\indexin{\nred}} & 0 \\ 0 & \singvals_{\indexout{\nred}}
\end{bmatrix}
\begin{bmatrix}
    \rightsings_{\indexin{\nred}} &
    \rightsings_{\indexout{\nred}}
\end{bmatrix}^\top,
\end{align*}
where $\Vlin = \leftsings_{\indexin{\nred}}=[\leftsing{j_1}, \dots, \leftsing{j_r}]$ contains the columns of $\leftsings$ included  by the greedy method and $\breve{\Vlin} = \leftsings_{\indexout{\nred}}$ includes the left-singular vectors with indices in $\indexin{\nred}$, i.e., the left-singular vectors that have not been selected by the greedy method. %
With this notation, we  restate the selection criterion of our greedy method given by the objective~\eqref{eq:Greedy:J} as
\begin{align}
  \hat{J}(\Vlin, \Vnonlin) = 
  \frobsq{\Vnonlin \featuremap(\Vlin^{\top}\snapshots) -
  \breve{\Vlin}\breve{\Vlin}^{\top}\snapshots} + \gamma \frobsq{\Vnonlin} \,.
  \label{eq:LstSqinSVDForm}
\end{align}
Minimizing \eqref{eq:LstSqinSVDForm} over $\Vlin$ and $\Vnonlin$ as in the greedy method in \eqref{eq:Greedy:GreedyOptiProblem} shows the greedy method %
seeks a basis matrix $\Vlin$ and a weight matrix $\Vnonlin$ such that $\Vnonlin h(\Vlin^{\top}\snapshots)$ approximates well $\breve{\Vlin}\breve{\Vlin}^{\top}\snapshots$, which is analogous to approximately satisfying \eqref{eq:Disc:VSWH} when setting $\snapshots_r = \Vlin^{\top}\snapshots$ and $\breve{\snapshots} = \breve{\Vlin}^{\top}\snapshots$. In particular,  
the weight matrix $\Vnonlin$ has to achieve two goals: First, its column space has to approximate well the subspace of the orthogonal complement of the basis matrix $\Vlin$ that is most important for approximating the data $\snapshots$ with respect to the error in the Frobenius norm.  
Second, the oblique projection $\Vnonlin^{+}\snapshots$ is approximated well by the lifted encoded data points $h(\snapshots_r)$. 

\newcommand{\features}{\featuremap(\singvals_{\indexin{\nred}}\rightsings_{\indexin{\nred}}^\top)}
The minimizer of~\eqref{eq:LstSqinSVDForm} is given by 
\begin{align*}
\Vnonlin=\breve{\Vlin}\breve{\Vlin}^{\top}\snapshots h(\Vlin^{\top}\snapshots)^{\top} (h(\Vlin^{\top}\snapshots)h(\Vlin^{\top}\snapshots)^{\top} + \gamma \boldsymbol{I})^{-1},
\end{align*}
and reveals that a necessary condition for a non-zero weight matrix $\Vnonlin$ is that the product 
\begin{equation}\label{eq:DiscussionC}
\bfC = \breve{\snapshots}h(\snapshots_r)^{\top}
\end{equation}
is non-zero. 
In particular, if the $\ell$-th row of $\bfC$ is (close to) zero, then the data points approximated on the manifold $\hat{\snapshots} = g_{\Vlin,\Vnonlin}(f_{\Vlin}(\snapshots))$ are (almost) orthogonal to the $\ell$-th column of $\breve{\Vlin}$ and thus the $\ell$-th column of $\breve{\Vlin}$ cannot be utilized to improve (substantially) upon the accuracy of the linear approximations in the space spanned by the columns of $\Vlin$.
Indeed, if all rows of $\bfC$ are zero, then the weight matrix $\Vnonlin$ is zero and thus the corresponding quadratic manifold collapses to the space spanned by the columns of the basis matrix $\Vlin$. Thus, for the quadratic manifold to achieve more accurate approximations than the space spanned by $\Vlin$, it is necessary that $\bfC$ contains non-zero rows.

We can interpret the matrix $\bfC$ as the unnormalized correlation matrix between the lifted encoded data points $h(\snapshots_r)$ and the coordinates $\breve{\snapshots}$ of the projections of the data points onto the orthogonal complement of $\Vlin$, which is in agreement with the discussion in Section~\ref{sec:Discussion:Correlation1}. Later in the numerical experiments, we will empirically show that using the leading $\nred$ left-singular vectors to form $\Vlin$ leads to many more rows in the unnormalized coefficient matrix $\bfC$ given in \eqref{eq:DiscussionC} that has magnitude close to zero than the left-singular vectors selected by the proposed greedy method.

\subsection{Nonlinear encoding}
\label{sec:nonlinear_encoding}
\newcommand{\gnobjective}{L_{GN}}
\newcommand{\jacres}{J_{GN}}
We follow \cite{GeelenWW2023Operator, SharmaMBGGK2023Symplectic, BarnettF2022Quadratic, BarnettFM2023Neural-network-augmented} and use linear encoder functions $f_{\Vlin}$ to find $\redstate = f_{\Vlin}(\fullstate)$ for a data point $\fullstate$. Other works such as \cite{JainTRR2017quadratic, RutzmoserRTJ2017Generalization,GeelenBW2023Learning} work with nonlinear encodes of $\fullstate$ onto $\redstate$. 
We now establish a connection between the linear encoding $f_{\Vlin}$ that we use and a nonlinear encoding that minimizes the reconstruction error. Let us consider the nonlinear least-squares problem to find $\hat{\fullstate} \in \mathcal{M}_{r}$ from $\fullstate$, 
\begin{align}
\label{eq:min_rec_error}
    \operatorname*{arg\,min}\limits_{\recoveredstate \in \mathcal{M}_r}\euclsq{\recoveredstate - \fullstate} \,.%
\end{align}
A solution $\hat{\fullstate}^*$  of \eqref{eq:min_rec_error} gives a point on the manifold $\mathcal{M}_r$ that is closest to the data point $\fullstate$. Equivalently to \eqref{eq:min_rec_error}, we can solve the unconstrained least-squares problem 
\begin{equation}\label{eq:eq:min_rec_errorG}
\operatorname*{arg\,min}_{\redstate \in \mathbb{R}^{\nred}} \euclsq{g_{\Vlin,\Vnonlin}(\redstate)-\fullstate}
\end{equation}
by using that any element $\hat{\fullstate}$ in $\mathcal{M}_r$ can be represented with an $\redstate \in \mathbb{R}^{\nred}$ via $\hat{\fullstate} = g_{\Vlin, \Vnonlin}(\redstate)$. 

Let us now consider the first-order optimality condition of \eqref{eq:eq:min_rec_errorG},  
\[
\bfJ_{g_{\Vlin, \Vnonlin}}(\redstate)^{\top} (g_{\Vlin,\Vnonlin}(\redstate) - \fullstate) = \boldsymbol 0\,,
\]
where $\bfJ_{g_{\Vlin, \Vnonlin}}(\redstate)$ is the Jacobian matrix of the decoder function $g_{\Vlin,\Vnonlin}$ with respect to $\redstate$,
\begin{equation}\label{eq:Prelim:JacobianG}
\bfJ_{g_{\Vlin, \Vnonlin}}(\redstate) = \Vlin + \Vnonlin h^{\prime}(\redstate) \in \mathbb{R}^{\nfull \times \nred}\,. %
\end{equation}
The function $h': \mathbb{R}^{\nred} \to \mathbb{R}^{r(r+1)/2 \times \nred}$ evaluates to the Jacobian matrix of the feature map $h$ at a point $\redstate$. Applying the Gauss-Newton method to numerically solve the nonlinear regression problem \eqref{eq:eq:min_rec_errorG} leads to the iterations
\begin{equation}\label{eq:QMDiscuss:GN}
\redstate^{(i + 1)} = \redstate^{(i)} + \delta\redstate^{(i)}\,, \qquad i = 0, 1, 2, 3, \dots\,,
\end{equation}
where the update $\delta\redstate^{(i)} \in \mathbb{R}^{\nred}$ at iteration $i$ is the minimal-norm solution of the linear regression problem
\[
\min_{\delta\redstate^{(i)} \in \mathbb{R}^{\nred}} \|\bfJ_{g_{\Vlin, \Vnonlin}}(\redstate^{(i)})\delta\redstate^{(i)} - g_{\Vlin,\Vnonlin}(\redstate^{(i)}) + \fullstate\|_2^2 + \|\delta \redstate^{(i)}\|_2^2\,.
\]

We can recover the linear encoder function $f_{\Vlin}$ by starting the Gauss-Newton iterations \eqref{eq:QMDiscuss:GN} at $\redstate^{(0)} = \boldsymbol 0 \in \mathbb{R}^{\nred}$ and stopping after $i = 1$ iterations, because this leads to the encoded point $\redstate = \Vlin^{\top}\fullstate = f_{\Vlin}(\fullstate)$.
Another perspective that motivates using the linear encoder function $f_{\Vlin}$ is that it sets the error $g_{\Vlin,\Vnonlin}(\redstate) - \fullstate$ orthogonal to the space spanned by the columns of $\Vlin$ rather than the space spanned by the Jacobian \eqref{eq:Prelim:JacobianG} of the decoder function $g_{\Vlin,\Vnonlin}$.
Thus, $f_{\Vlin}(\fullstate)$ is the solution to the problem
\[
\operatorname*{arg\,min}_{\redstate \in \mathbb{R}^{\nred}} \|\Vlin^{\top}(\decoder_{\Vlin,\Vnonlin}(\redstate) - \fullstate)\|_2^2\,.
\]
While we propose to use the computationally convenient linear encoder $f_{\Vlin}$; in the numerical examples, we compare the linear encoding obtained with $f_{\Vlin}$ to the nonlinear encoding obtained with running Gauss-Newton iterations for $i > 1$ steps.
}

\section{Numerical experiments}
\label{sec:numerical_experiments}
We demonstrate the greedy method on four different data sets.
The first data set represents an advecting wave, which is challenging to reduce with linear methods such as PCA. The same example is used as benchmark  in~\cite{GeelenWW2023Operator}. We then consider two data sets that describe more complicated wave behavior such as nonlinear waves and interacting pulse signals. 
The fourth example demonstrates the greedy construction of quadratic manifolds on a data set that describes a turbulent flow in a channel.

\subsection{Setup}\label{sec:NumExp:Setup}
We compare three approaches for constructing quadratic manifolds. The first one follows \cite{GeelenWW2023Operator,BarnettF2022Quadratic} and uses the leading $\nred$ left-singular vectors of the data matrix to span the subspace $\Vcal$ for the linear approximation. 
The weight matrix $\Vnonlin$ is fitted via the linear least-squares problem \eqref{eq:Prelim:LinearLSQApproach} using a training data matrix. The regularization parameter $\gamma$ for \eqref{eq:Prelim:LinearLSQApproach} is chosen from $\{10^{-8}, 10^{-7}, \dots, 10^{-2}\}$ so that it minimizes the objective of \eqref{eq:Prelim:LinearLSQApproach} on a validation data set. The second approach is based on alternating minimization as introduced in \cite{GeelenBW2023Learning}, which fits $\Vlin$ and $\Vnonlin$ via an alternating minimization scheme; see Appendix~\ref{appx:AlternatingMin} for the technical details of this approach. The alternating minimization approach depends on a range of hyper-parameters, which we discuss in Appendix~\ref{appx:AlternatingMin}.  The third approach is the proposed greedy construction, where we set $\nconsider=10\nred_{\text{max}}$, where $\nred_{\text{max}}$ is the largest reduced dimension considered in the respective experiment. We note that we perform the greedy step in line 5 of Algorithm~\ref{alg:ourquadmani} over the left-singular vectors with indices $1, \dots, \nconsider + i$ at greedy iteration $i$, instead of only up to $\nconsider$, so that the number of evaluations of the objective~\eqref{eq:Jprime} is constant $\nconsider$ over all greedy iterations $i = 1, \dots, \nred$.  The regularization parameter $\gamma$ is obtained via a grid search over $10^{-8}, 10^{-7}, \dots, 10^{-2}$, where we then use the one that minimizes the objective \eqref{eq:Greedy:GreedyOptiProblem} on a validation data set as in the first approach.
The feature map $h$ is fixed to the map defined in~\eqref{eq:condensed_kronecker}, independent of which approach is used to fit $\Vlin$ and $\Vnonlin$.

We report the relative error of approximating test data points that were not used during training or for validation. The relative error is computed as
\begin{equation}\label{eq:NumExp:RelErr}
  E_{\mathrm{rel}}(\snapshots^{(\mathrm{test})}) = \frac{1}{\left\|\snapshots^{(\mathrm{test})}\right\|_F}\left\|\lift(\embed(\snapshots^{(\mathrm{test})})) - \snapshots^{(\mathrm{test})}\right\|_F,
\end{equation}
where $g$ and $f$ are decoding and encoding maps, respectively, and $\snapshots^{(\mathrm{test})}$ is the test data matrix. 

\changed{
We will also report quantities to provide insights following the discussion in Section~\ref{sec:discussion_correlation}. Let $\Vlin \in \mathbb{R}^{\nfull \times \nred}$ be the basis matrix that contains the left-singular vectors with indices $i_1, \dots, i_{\nred}$ and let $\breve{\Vlin} \in \mathbb{R}^{\nfull \times (\nsnapshots-\nred)}$ contain the remaining left-singular vectors with indices $\{1, \dots, \nsnapshots\}\setminus\{i_1, \dots, i_{\nred}\}$ of a data matrix $\snapshots \in \mathbb{R}^{\nfull \times \nsnapshots}$ with $\nsnapshots \leq \nfull$. Following Section~\ref{sec:discussion_correlation}, we consider the matrix $\breve{\Vlin}^{\top}\snapshots$ and subtract the mean over the rows to obtain $\overline{\breve{\Vlin}^{\top}\snapshots}$. Analogously, we take the matrix $h(\Vlin^{\top}\snapshots)$ and subtract the row mean to obtain $\overline{h(\Vlin^{\top}\snapshots)}$.  We then will consider the normalized counterpart $\tilde{\bfC}\in \mathbb{R}^{\nsnapshots \times \nredmod}$ to the matrix $\bfC$ defined in \eqref{eq:DiscussionC}, with entries
\begin{equation}\label{eq:correlations_compute}
\tilde{\bfC}_{ij} = \frac{\left[\overline{(h(\Vlin^{\top}\snapshots))}\right]_{j, :}\,\,\left[\overline{(\breve{\Vlin}^{\top}\snapshots)}\right]_{i, :}}{\|[\breve{\Vlin}^{\top}\snapshots]_{i, :}\|_2 \|[h(\Vlin^{\top}\snapshots)]_{j, :}\|_2} \,,
\end{equation}
where $[\cdot]_{i, :}$ and $[\cdot]_{j, :}$ select the $i$-th row and the $j$-th row of the matrix argument, respectively. 
We interpret $\tilde{\bfC}$ as the correlation matrix between the lifted encoded data points given by $h(\Vlin^{\top}\snapshots)$ and the coordinates $\breve{\Vlin}^{\top}\snapshots$ of representing the data points in the basis spanned by the left-singular vectors that are not selected for $\Vlin$. 
We will compare the matrices $\tilde{\bfC}$ corresponding to quadratic manifolds constructed from the leading $\nred$ left-singular vectors only and quadratic manifolds constructed with the proposed greedy method.
}

In the following, data matrices are centered so that their row-wise mean is zero. We apply the same shift to the validation and test data. The train, validation, and test data  are available at \url{https://zenodo.org/records/10738062}. The experiments are run on four CPU cores of an Intel Xeon Platinum 8268 CPU. The nonlinear advection diffusion experiments were run on four cores of an Intel Xeon Platinum 8470QL CPU. The memory allocation is chosen depending on the data size of the example. In all examples, all methods have access to the same amount of memory.

\begin{figure}
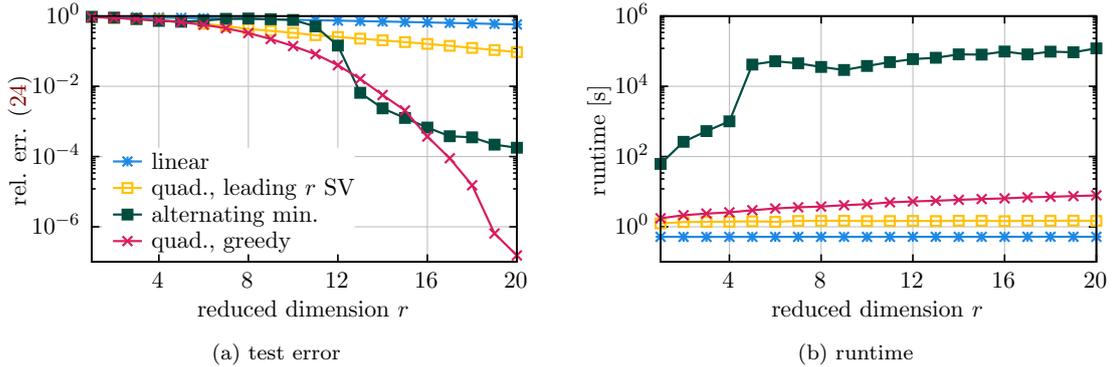

  \centering
  \begin{tabular}{cc}
  \resizebox{0.49\columnwidth}{!}{\input{./PlotSources/linear_transport_rel_errs_method_comparison.tex}} &
  \resizebox{0.49\columnwidth}{!}{\input{./PlotSources/linear_transport_runtimes_method_comparison.tex}} \\
  \scriptsize{(a) test error}                                                     &
  \scriptsize{(b) runtime}
  \end{tabular}
  \caption{Advecting wave: The proposed greedy approach achieves up to five orders of magnitude higher accuracy than using the leading $\nred$ left-singular vectors for the quadratic manifold construction. Additionally, the greedy approach incurs an orders of magnitude lower runtime than alternating minimization.} 
\label{fig:linear_transport_comparison}
\end{figure}

\subsection{Approximating advecting waves}\label{sec:NumExp:LinAdv}
We follow \cite{GeelenWW2023Operator} and consider the Gaussian bump function
\begin{equation}
s_0(x) = \frac{1}{\sqrt{0.0002\pi}} \exp\left(-\frac{(x-\mu)^2}{0.0002}\right), \quad x \in \R,
\end{equation}
with $\mu=0.1$ and then shift its mean in the spatial domain $[0, 1] \subset \mathbb{R}$ as $s(t, x) = s_0(x - ct)$, where $t \in [0, 0.1]$ and $c = 10$. Notice that $s$ is the solution to an instance of the linear advection equation \cite{GeelenWW2023Operator}. We generate a matrix $ [\fullstatei{1}, \dots, \fullstatei{2000}]\in \R^{4096 \times 2000}$ by evaluating $s$ at times $t_i = 0.2(i-1)/2000$ for $i=1, \dots, 2000$ over 4096 equidistant $x_1, \dots, x_{4096}$ in $[0, 1]$.
We then construct the training data matrix as $\snapshots^{\text{(train)}}=[\fullstatei{1}, \fullstatei{3}, \dots, \fullstatei{1999}]$, the validation data matrix as $\snapshots^{\text{(val)}}=[\fullstatei{2}, \fullstatei{6}, \dots, \fullstatei{1998}]$, and the test data matrix as $\snapshots^{\text{(test)}}=[\fullstatei{4}, \fullstatei{8}, \dots, \fullstatei{2000}]$
The regularization parameter for fitting $\Vnonlin$ is $\gamma=10^{-8}$ for the greedy approach and also for the approach using the leading $\nred$ left-singular vectors. For the alternating minimization algorithm, the regularization parameter is $\gamma=10^{-4}$.

\begin{figure}
  \resizebox{0.99\columnwidth}{!}{\input{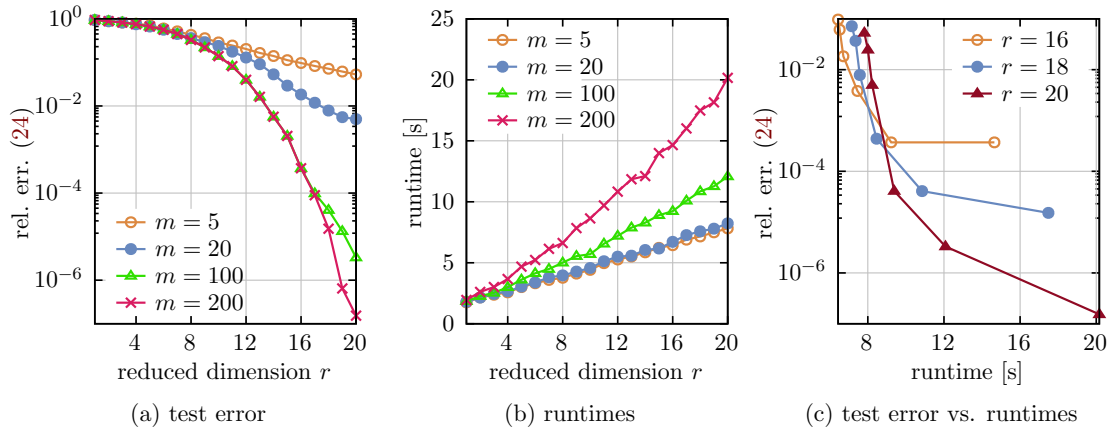}}
  \caption{Advecting wave: Runtime and accuracy of the greedy method can be traded off by varying the number $\nconsider$ of left-singular vectors that are considered in each greedy iteration.}
  \label{fig:linear_transport_greedy_maxdimconsider}
\end{figure}
  
In Figure~\ref{fig:linear_transport_comparison}a, we compare the relative %
error \eqref{eq:NumExp:RelErr} of approximating the test data on the quadratic manifolds obtained with three approaches as well as just the linear approximation error of the test data in the subspace spanned by the leading $\nred$ left-singular vectors of $\snapshots^{\text{(train)}}$. In agreement with the results in \cite{GeelenWW2023Operator}, quadratic manifolds obtained with setting $\Vcal$ to the subspace spanned by the leading $\nred$ left-singular vectors (see Section~\ref{subsubsec:prev_polynomial_feature_map}) leads to a lower relative error~\eqref{eq:NumExp:RelErr} on test data than the linear approximations in $\Vcal$ alone. The alternating minimization approach finds a quadratic manifold that achieves about two orders of magnitude lower test errors. 
The proposed greedy approach constructs a quadratic manifold that achieves an about three orders of magnitude lower relative error than alternating minimization and an almost five orders of magnitude lower error than when setting $\Vcal$ to the leading $\nred$ left-singular vectors of $\snapshots^{(\text{train})}$.  Let us now consider the computational costs of three approaches; see Figure~\ref{fig:linear_transport_comparison}b. The proposed greedy method incurs an orders of magnitude lower wallclock runtime than the alternating minimization approach. Notice that the costs of the hyper-parameter sweeps are not included in the runtime in Figure~\ref{fig:linear_transport_comparison}, which are substantial for alternating minimization compared to the greedy approach.
In Figure~\ref{fig:linear_transport_greedy_maxdimconsider}, we show that a lower runtime of the greedy method can be traded-off with a higher error by varying the number $\nconsider$ of the left-singular vectors that are considered during the greedy iterations.
\changed{To demonstrate this, Figure~\ref{fig:linear_transport_greedy_maxdimconsider}a shows the relative error for different choices of $\nconsider$ and Figure~\ref{fig:linear_transport_greedy_maxdimconsider}b shows the corresponding runtime as $\nred$ increases. Finally, Figure~\ref{fig:linear_transport_greedy_maxdimconsider}c shows the relative error over the runtime for different reduced dimensions $r \in \{16, 18, 20\}$. 
}

\changed{
Figure~\ref{fig:linear_transport_extra_errors}a shows the singular value decay of the training data matrix and Figure~\ref{fig:linear_transport_extra_errors}b shows a comparison of the relative error between using the linear encoder $f_{\Vlin}$ and the nonlinear encoding via the nonlinear least-squares problem \eqref{eq:min_rec_error} solved with the Gauss-Newton method as discussed in Section~\ref{sec:nonlinear_encoding}. The maximum number of Gauss-Newton iterations is 20 and we stop early when the relative error changes by less than $10^{-12}$. We also plot the lower bound derived in Section~\ref{sec:lower_bounds}. Finding an encoding via the nonlinear least-squares problem only leads to small improvements in accuracy. 
}

\changed{
In Figure~\ref{fig:linear_transport_correlations}, we show the magnitudes of the entries of the top $\nredmod = 210$ rows of $\tilde{\bfC}$, which is computed as in~\eqref{eq:correlations_compute} for $\nred=20$. Large values mean that the lifted encoded data points $h(\Vlin^{\top}\snapshots)$ are strongly correlated to the coordinates of projections $\breve{\Vlin}^{\top}\snapshots$ of the data points onto the left-singular vectors, which is a necessary condition for the quadratic correction term to be effective; see Section~\ref{sec:discussion_correlation}.  As shown in Figure~\ref{fig:linear_transport_correlations}a, when choosing the first $\nred$ left-singular vectors only for forming $\Vlin$, then there is only low  correlation between the lifted encoded data points and coordinates corresponding to left-singular vectors of index larger than 50. This indicates that the lifted encoded data points are not informative for approximating the coordinates corresponding to later left-singular vectors and thus the quadratic correction term is less effective. In contrast, our greedy method selects earlier and later left-singular vectors to form the basis matrix $\Vlin$ of the quadratic manifold that lead to strong correlation even with coordinates corresponding to  later left-singular vectors, as indicated by the non-zero magnitudes of the entries in rows larger than 50 in the matrix $\tilde{\bfC}$ shown in Figure~\ref{fig:linear_transport_correlations}b. This indicates that the quadratic manifold obtained with the greedy method can better leverage the quadratic correction term, which aligns well with the other results shown in this section and the discussion in Section~\ref{sec:discussion_correlation}. %
}

\begin{figure}
    \centering
    \resizebox{0.99\columnwidth}{!}{\input{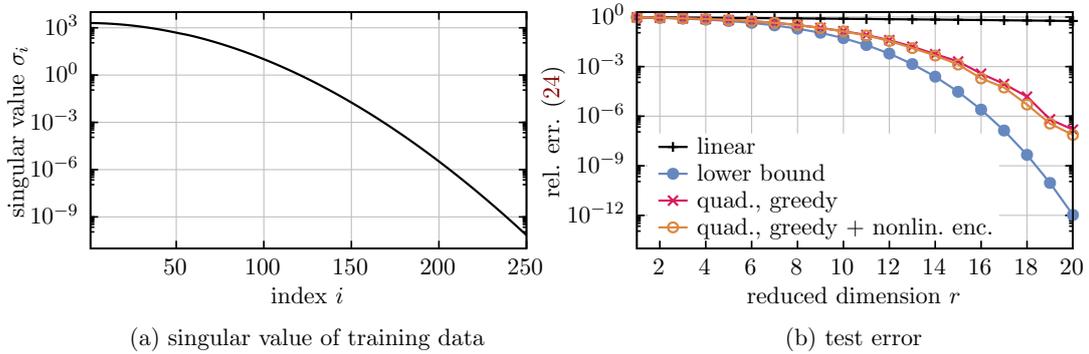}}
    \caption{Advecting wave: Singular value decay of training data set and relative error comparison to lower bound and nonlinear encoding.}
    \label{fig:linear_transport_extra_errors}
\end{figure}

\begin{figure}
    \centering
    \resizebox{0.99\columnwidth}{!}{\input{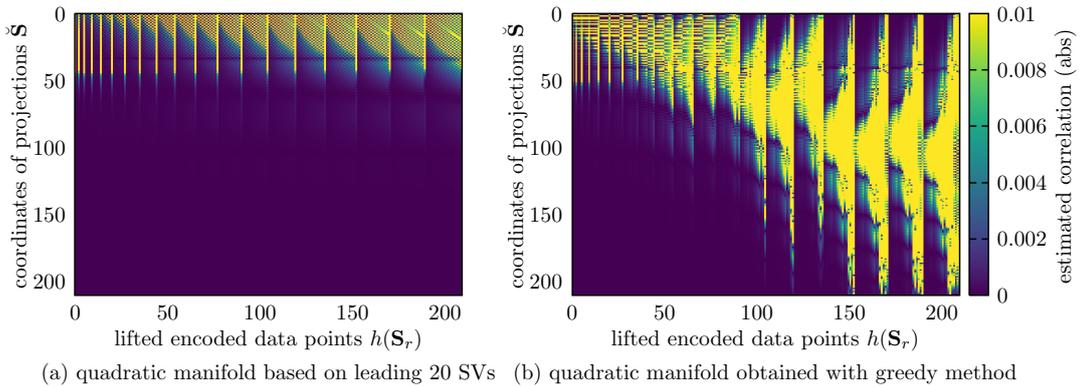}}
    \caption{Advecting wave: The plots show the magnitude of the entries of the correlation matrix \eqref{eq:correlations_compute}. It is necessary that the lifted encoded data points are correlated in the sense of \eqref{eq:correlations_compute} with the coordinates corresponding the later left-singular vectors for the quadratic correction term to be effective; see Section~\ref{sec:discussion_correlation}. Plot (a) shows that a basis matrix $\Vlin$ using the first $r = 20$ left-singular vectors only lead to lifted encoded data points that are poorly correlated with the coordinates of projections onto later left-singular vectors (lower rows). In contrast, the left-singular vectors selected by the proposed greedy method achieve stronger correlation as shown in plot (b) and in alignment with the other results in this section.} %
    \label{fig:linear_transport_correlations}
\end{figure}

\begin{figure}
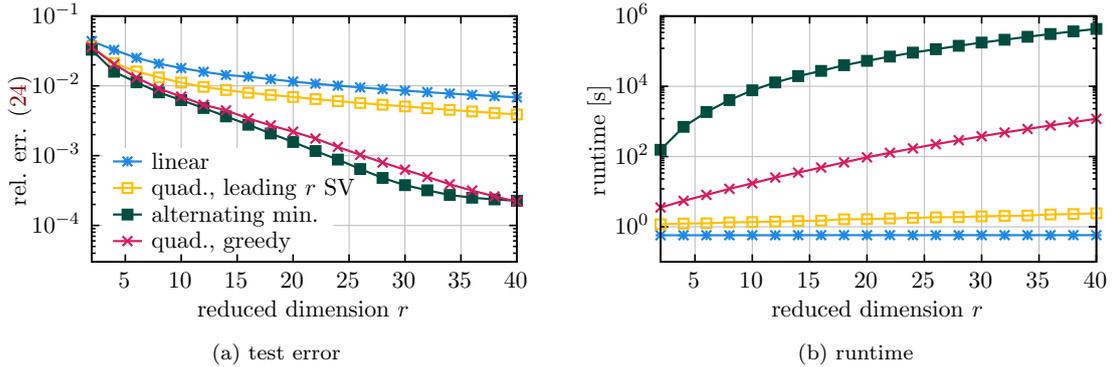

  \begin{tabular}{cc}
  \resizebox{0.49\columnwidth}{!}{\input{./PlotSources/burgers_rel_errs_method_comparison.tex}} &
  \resizebox{0.49\columnwidth}{!}{\input{./PlotSources/burgers_runtimes_method_comparison.tex}} \\
    \scriptsize{(a) test error}&
    \scriptsize{(b) runtime}
  \end{tabular}
  \caption{Nonlinear advection-diffusion: The proposed greedy method achieves an at least one order of magnitude lower error than using just the $\nred$ leading singular values for the quadratic manifold construction. The manifold found with alternating minimization achieves a comparable error but the runtime of alternating minimization is up to four orders of magnitude higher than the runtime of the proposed greedy method. And, alternating minimization requires extensive hyper-parameter tuning.} 
  \label{fig:burgers_example}
\end{figure}

\subsection{Waves described by nonlinear advection--diffusion processes}
Let us now consider waves that advect and change their shapes. The training data matrix is $\snapshots^{\text{(train)}} \in \R^{5000 \times 2000}$ and contains as columns data points that represent the waves; %
details of the data generation are described in Section~\ref{app:burgers_setup}.
The validation and test data matrices $\snapshots^{\text{(val)}} \in \R^{5000 \times 250}$ and $\snapshots^{\text{(test)}} \in \R^{5000 \times 250}$ consist of 250 data points that different from the data points in the training data.  

We compare in Figure~\ref{fig:burgers_example}a the accuracy obtained with the three methods discussed in Section~\ref{sec:NumExp:Setup} and the accuracy of the linear approximation in the subspace spanned by the $\nred$ leading left-singular vectors of $\snapshots^{\text{(train)}}$.
The proposed greedy approach achieves an about one order of magnitude lower error than just using the leading $\nred$ left-singular vectors for the space $\Vcal$. The error achieved by the greedy method is comparable to the error achieved by alternating minimization; however, alternating minimization is orders of magnitude more expensive, as is shown in Figure~\ref{fig:burgers_example}b. We use the regularization parameter $\gamma = 10^{-2}$ for alternating minimization, $\gamma=10^{-2}$ for the quadratic manifold based on the leading $\nred$ left-singular vectors, and $\gamma = 10^{-3}$ for the proposed greedy approach; all of these parameters were obtained with the hyper-parameter tuning as described in Section~\ref{sec:NumExp:Setup}.

In Figure~\ref{fig:burgers_reg_sensitivity}, we demonstrate that the proposed greedy method is robust against the regularization parameter $\gamma$. Whereas just using the leading $\nred$ left-singular vectors leads to quadratic manifolds with widely different performance for varying $\gamma$, the greedy method shows comparable performance for the range $\gamma \in [10^{-6},10^{-2}]$, which indicates its robustness in terms of hyper-parameter tuning. 

\begin{figure}
  \resizebox{0.99\columnwidth}{!}{\input{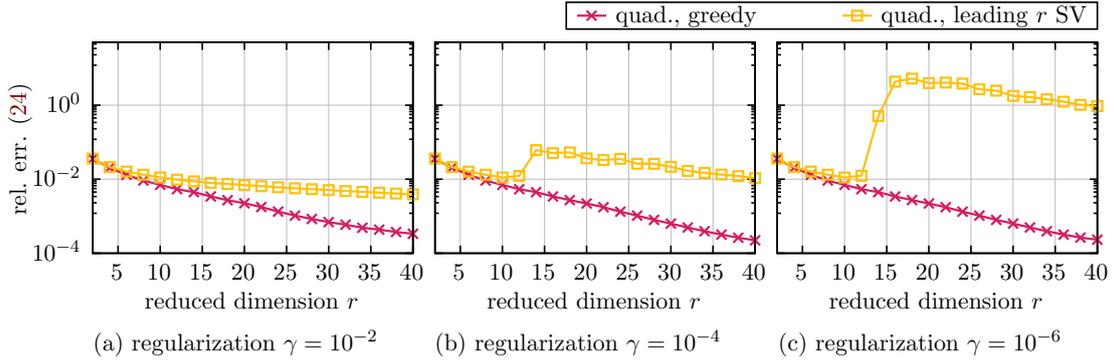}}
  \caption{Nonlinear advection-diffusion: The proposed greedy method is robust with respect to the regularization parameter~$\gamma$, whereas using the leading $\nred$ singular vectors for the manifold construction shows higher sensitivity to $\gamma$ in terms of error on the test data.}
  \label{fig:burgers_reg_sensitivity}
\end{figure}

\begin{figure}
    \centering
    \resizebox{0.99\columnwidth}{!}{\input{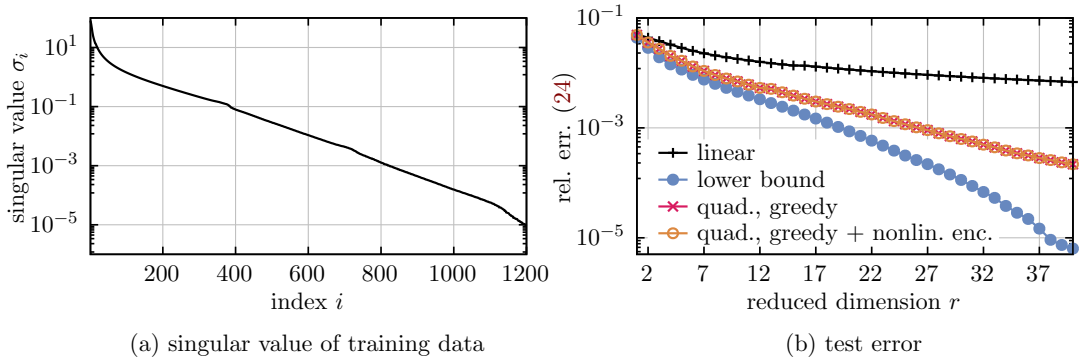}}
    \caption{Nonlinear advection diffusion: Singular value decay of training data set and relative error comparison to lower bound and nonlinear encoding.}
    \label{fig:burgers_extra_errors}
\end{figure}

\begin{figure}
    \centering
    \resizebox{0.99\columnwidth}{!}{\input{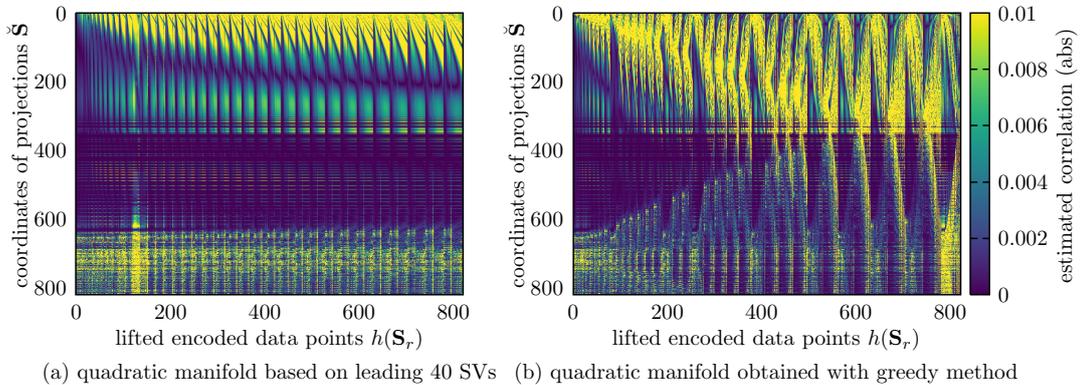}}
    \caption{Nonlinear advection diffusion: The lifted encoded data points corresponding to the greedily constructed quadratic manifold are higher correlated in the sense of \eqref{eq:correlations_compute} with the coordinates of projections onto left-singular vectors than selecting only the first $\nred$ left-singular vectors for $\Vlin$. A high correlation is a necessary condition for the quadratic correction term to be effective; see Section~\ref{sec:discussion_correlation}.} %
    \label{fig:burgers_correlations}
\end{figure}

\changed{
Figure~\ref{fig:burgers_extra_errors}a shows the singular value decay of the training data matrix and Figure~\ref{fig:burgers_extra_errors}b shows a comparison of the relative error between using a linear encoder and the nonlinear encoding obtained via the nonlinear least-squares problem \eqref{eq:min_rec_error} solved by the Gauss-Newton method; see  Section~\ref{sec:nonlinear_encoding}. We use the same hyper-parameters for the Gauss-Newton method as in the previous example. We also show the lower bound derived in Section~\ref{sec:lower_bounds}. In this example, the nonlinear encoding only leads to slight improvements in accuracy. 
}

\changed{
In Figure~\ref{fig:burgers_correlations}, we plot the magnitude of the entries of the top $\nredmod = 820$ rows of the correlation matrix $\tilde{\bfC}$ defined in \eqref{eq:correlations_compute} for $\nred=40$; see the previous example in Section~\ref{sec:NumExp:LinAdv} for detailed descriptions. Using the first $\nred$ left-singular vectors to form $\Vlin$ (plot (a)) leads to less correlation than using our greedy method for selecting $\Vlin$ (plot (b)), which is in agreement with the higher accuracy achieved by the quadratic manifold obtained with our greedy method for this data set. 
}

\subsection{Hamiltonian interacting pulse signals}
We now consider a pulse signal traveling in a two-dimensional domain, which is governed by the Hamiltonian wave equation. We impose periodic boundary conditions so that the initial pulse spreads out and then interacts with the pulse, leading to interaction patterns. %
We consider the velocities in x and y direction and density component on a $600 \times 600$ grid so that the dimension of the data points is $\nfull = 1.08 \times 10^{6}$; details of generating the data are given in Appendix~\ref{app:wave_setup}. We generate a matrix $[\fullstatei{1}, \dots, \fullstatei{1600}] \in \R^{1080000\times 1600}$ by sampling the velocities and density over time and split the columns into the training data matrix $\snapshots^{\text{(train)}}=[\fullstatei{1}, \fullstatei{3}, \dots, \fullstatei{1599}] \in \R^{108000\times 800}$, the validation data matrix $\snapshots^{\text{(val)}}=[\fullstatei{2}, \fullstatei{6}, \dots, \fullstatei{1598}]\in\R^{108000\times 400}$, and the test data matrix $\snapshots^{\text{(test)}}=[\fullstatei{4}, \fullstatei{8}, \dots, \fullstatei{1598}]\in\R^{108000\times 400}$. 

The high dimension of the data points means that alternating minimization becomes computationally intractable; see the previous example with dimension 4096 where alternating minimization already took almost eleven days wallclock time. We therefore only compare the linear approximation given by PCA and the quadratic manifolds obtained with the leading $\nred$ singular vectors and our greedy method. The regularization parameter were found to be $\gamma=10^{-8}$ for both quadratic-manifold methods.

The relative errors and runtimes are shown in Figure~\ref{fig:hamiltonian_wave_rel_errs}. The greedy method achieves an accuracy improvement of eight orders of magnitude compared to using the leading $\nred$ singular vectors for the quadratic manifold construction. At the same time, the runtime of the two methods for constructing quadratic manifolds is comparable: the runtime is dominated by computing the SVD of the training data matrix, which has to be done once in both methods as discussed in Section~\ref{sec:Greedy:RuntimeImprovement}. The point-wise errors are plotted in  Figure~\ref{fig:hamiltonian_wave_snapshots_errors}, where one can see that the quadratic manifold obtained with the leading $\nred$ singular values leads to visible oscillations in the approximation whereas such oscillations are not visible in the approximation obtained with the quadratic manifold constructed with the proposed greedy method. 
\changed{
The singular value decay of the training data matrix as well as the comparisons to the lower bound that is proved in Proposition~\ref{prop:lower_bound} are shown in Figure~\ref{fig:hamiltonian_wave_extra_errors}. We also plot in Figure~\ref{fig:hamiltonian_wave_extra_errors} the error achieved when data points are encoded by solving the nonlinear least-squares problem \eqref{eq:min_rec_error} via the Gauss-Newton method. The hyper-parameters of the Gauss-Newton method are the same as in the previous examples. In agreement with previous results, only slightly lower errors are achieved with the nonlinear encoding.}

\changed{
Figure~\ref{fig:hamiltonian_wave_correlations} shows the magnitude of the entries of the correlation matrix \eqref{eq:correlations_compute} corresponding to the quadratic manifold obtained by using the first $\nred$ singular vectors only (plot (a)) and by our greedy method (plot (b)). The dimension is $\nred = 20$. In agreement with previous examples and the other results shown for this example, the lifted encoded data points are stronger correlated with the coordinates corresponding to later left-singular vectors when the basis matrix $\Vlin$ is selected by the proposed greedy method than when it contains the first $\nred$ leading left-singular vectors only. Recall that correlation is a necessary condition for the quadratic manifold to achieve higher approximation accuracy than the linear approximation space spanned by the columns of $\Vlin$; see Section~\ref{sec:discussion_correlation}. 
}

\begin{figure}
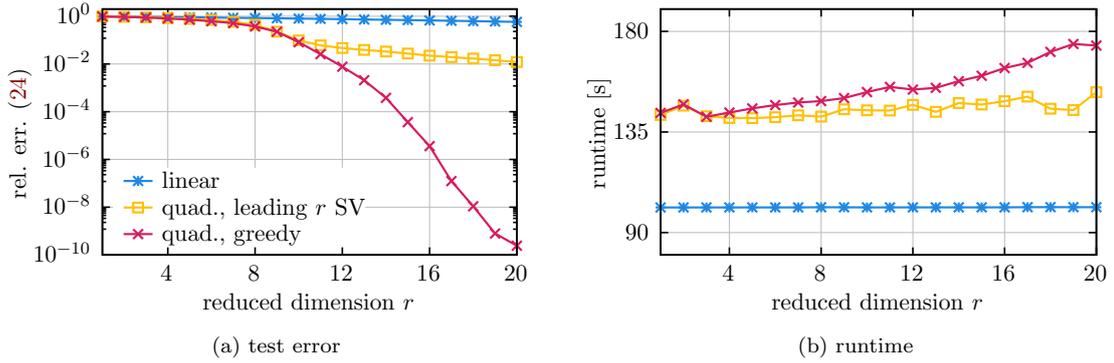

  \begin{tabular}{cc}
      \resizebox{0.49\columnwidth}{!}{\input{./PlotSources/hamiltonian_rel_errs_method_comparison.tex}} &
      \resizebox{0.49\columnwidth}{!}{\input{./PlotSources/hamwave_runtimes_method_comparison.tex}} \\
    \scriptsize (a) test error                                                                   &
    \scriptsize (b) runtime
  \end{tabular}
  \caption{Pulse signal: The proposed greedy approach for constructing quadratic manifolds leads to eight orders of magnitude more accurate approximations than quadratic manifolds based on the leading $\nred$ singular vectors. The runtime of the greedy approach is comparable to the runtime of the manifold construction using the leading $\nred$ singular vectors because the runtime is dominated by the SVD of the training data matrix, which has to be computed once in both methods; see Section~\ref{sec:Greedy:RuntimeImprovement}.}
  \label{fig:hamiltonian_wave_rel_errs}
\end{figure}

\begin{figure}
    \centering
      \resizebox{0.95\columnwidth}{!}{\input{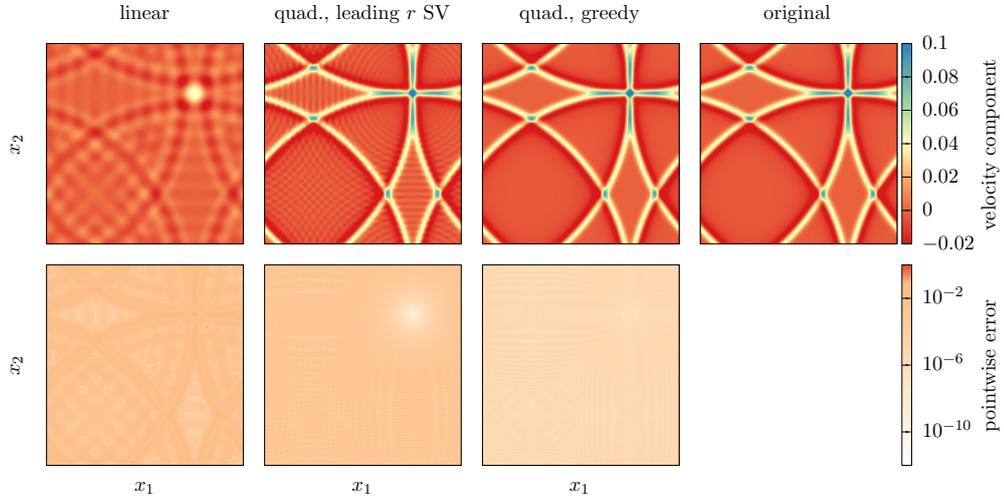}}
    \caption{Pulse signal: The quadratic manifold obtained with the leading $\nred$ singular values leads to visible oscillations in the approximation whereas such oscillations are not visible in the approximation obtained with the quadratic manifold constructed with the proposed greedy method.} 
    \label{fig:hamiltonian_wave_snapshots_errors}
\end{figure}

\begin{figure}
    \centering
    \resizebox{0.99\columnwidth}{!}{\input{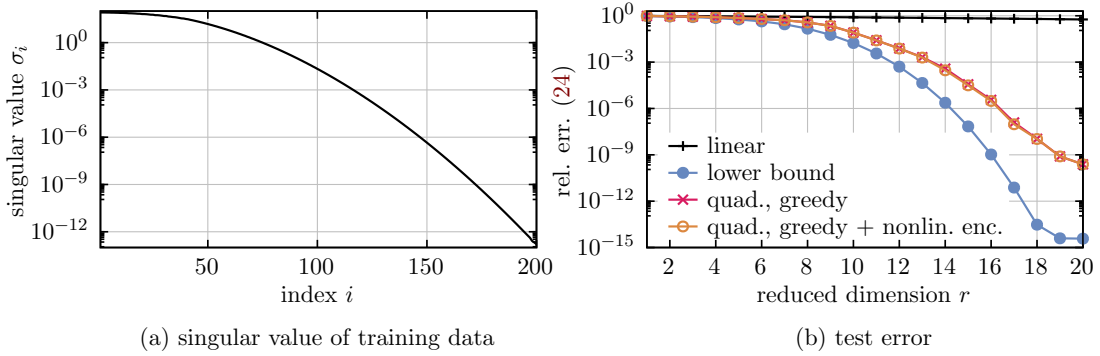}}
    \caption{Pulse signal: Singular value decay of training data set and relative error comparison to lower bound and nonlinear encoding.}
    \label{fig:hamiltonian_wave_extra_errors}
\end{figure}

\begin{figure}
    \centering
    \resizebox{0.99\columnwidth}{!}{\input{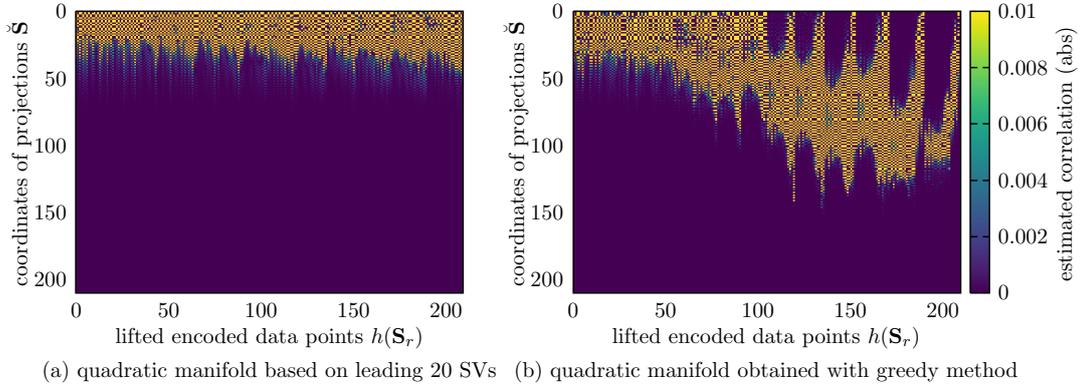}}
    \caption{Pulse signal: The lifted encoded data points on the quadratic manifold obtained with the proposed greedy method (plot (b)) leads to stronger correlation than the quadratic manifold obtained via the first $\nred$ left-singular vectors only.}
    \label{fig:hamiltonian_wave_correlations}
\end{figure}

\begin{figure}
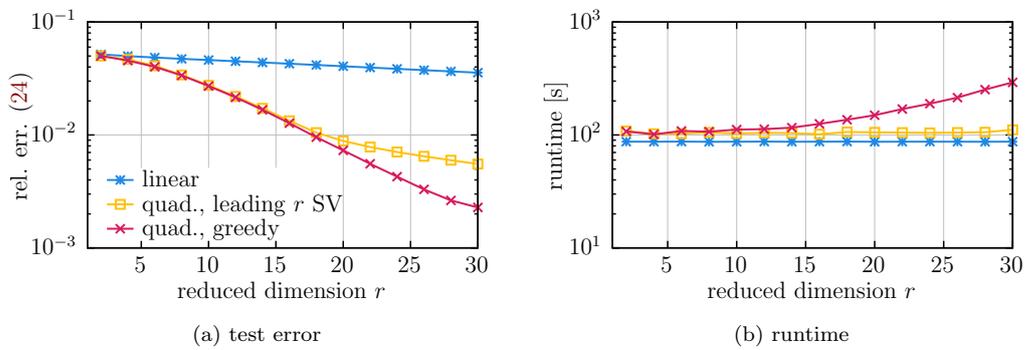

  \begin{tabular}{cc}
      \resizebox{0.45\columnwidth}{!}{\large\input{./PlotSources/channel_flow_rel_errs_method_comparison.tex}} &
      \resizebox{0.45\columnwidth}{!}{\large\input{./PlotSources/channel_flow_runtimes_method_comparison.tex}} \\
    \scriptsize (a) test error &
    \scriptsize (b) runtime
  \end{tabular}
  \caption{Turbulent flow: The greedy approach to constructing quadratic manifolds achieves about 50\% lower errors than using the leading $\nred$ singular vectors. The runtime of the greedy method is higher than using the leading $\nred$ singular vectors but remains within minutes for constructing the manifold in this example.}
  \label{fig:channel_flow_errs}
\end{figure}

\begin{figure}
    \centering
    \resizebox{0.95\columnwidth}{!}{\large\input{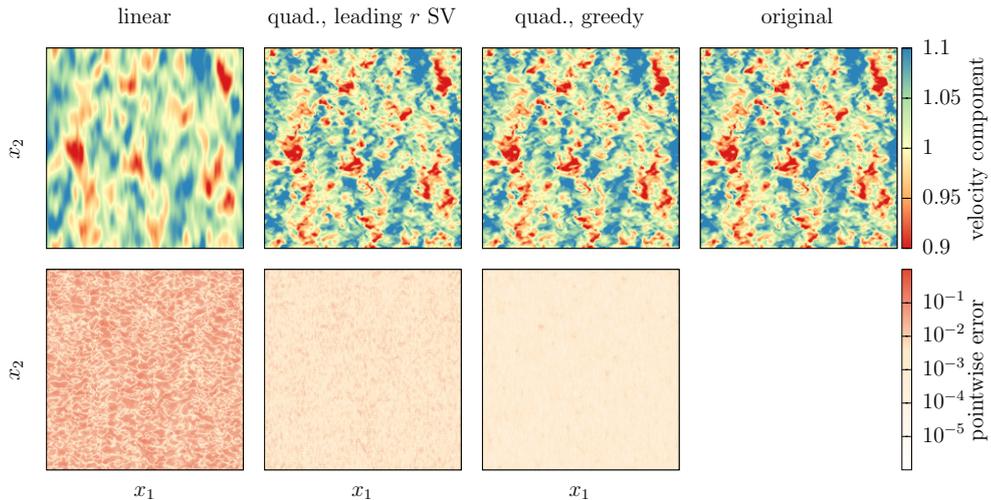}}
    \caption{Turbulent flow: The greedy method leads to a quadratic manifold with visibily lower point-wise error than using the leading $\nred$ singular vectors for constructing the quadratic manifold.}
    \label{fig:channel_flow_reconstructions}
\end{figure}

\subsection{Channel flow data}
We now consider a data set that represents the velocity field of a turbulent channel flow, which has been obtained wih the AMR-Wind simulation code~\cite{BrazellAVCSEH2021AMR-Wind}. The data comes from a wall-modeled large eddy simulation at Reynolds number $5200$, discretized using a staggered finite volume method into $384 \times 192 \times 32$ cells; we refer to the AMR-Wind simulation code for details~\cite{BrazellAVCSEH2021AMR-Wind}. The dimension is $\nfull = 384 \times 192 \times 32 = 2359296$. We have 1200 data points in total, which we split into training, validation, and test data as in the previous examples. 

The accuracy and runtime results are shown in Figure~\ref{fig:channel_flow_errs}. The approximations of the test data obtained with the quadratic manifolds have higher accuracy than linear approximations. The proposed greedy approach leads to a more accurate quadratic manifold than the manifold based on the leading $\nred$ singular vectors. The runtime of the greedy method is at least one order of magnitude higher than using the leading $\nred$ singular vectors; however, note that the runtime of the greedy is still minutes for constructing the quadratic manifold. The regularization parameter is set to $\gamma=10^{-2}$. In Figure~\ref{fig:channel_flow_reconstructions}, we show the approximations and their point-wise errors of a test data point and dimension $\nred=30$. For visualization, we show a $384 \times 192$-dimensional slice through the center of the channel flow field.
In agreement with the errors shown in Figure~\ref{fig:channel_flow_errs}a, the point-wise error of the approximation obtained with the greedy method is visibly lower than when using the leading $\nred$ singular vectors. 

\changed{
Figure~\ref{fig:channelflow_extra_errors}a shows the singular value decay of the training data matrix. Figure~\ref{fig:channelflow_extra_errors}b compares using the linear encoder to the nonlinear encoding obtained via the nonlinear least-squares problem \eqref{eq:min_rec_error}, which we numerically solve with the Gauss-Newton method as in the previous examples. We also plot the lower bound derived in Proposition~\ref{prop:lower_bound}. In agreement with the previous results, the nonlinear encoding helps to improve accuracy only slightly. 
}

\changed{
In Figure~\ref{fig:channelflow_correlations}, we plot the magnitude of the entries of the correlation matrix in~\eqref{eq:correlations_compute} up to row $\nredmod = 465$ for dimension $\nred=30$. In agreement with previous results, the correlation between the lifted encoded data points and the coordinates corresponding to later left-singular vectors is stronger for the quadratic manifold constructed with the proposed greedy method than the manifold obtained by using the first $\nred$ left-singular vectors only; we refer to Section~\ref{sec:discussion_correlation} for an in-depth discussion about the importance of the correlation. 
}

\begin{figure}
    \centering
    \resizebox{0.99\columnwidth}{!}{\input{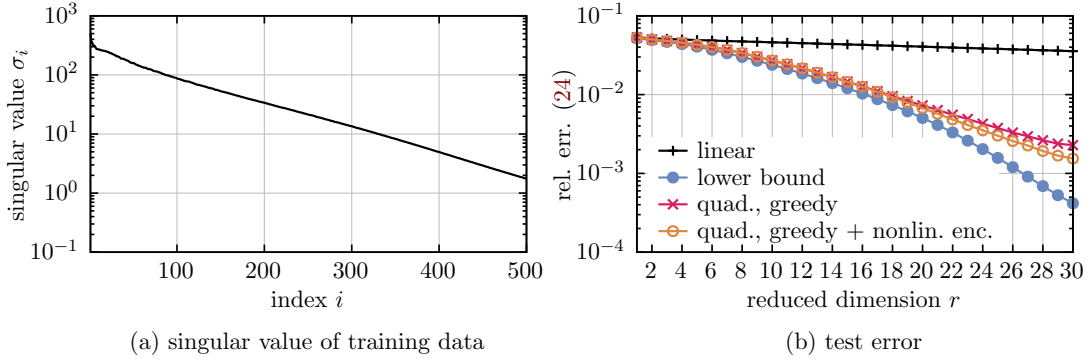}}
    \caption{Turbulent flow: Singular value decay of training data set and relative error comparison to lower bound and nonlinear encoding.}
    \label{fig:channelflow_extra_errors}
\end{figure}

\begin{figure}
    \centering
    \resizebox{0.99\columnwidth}{!}{\input{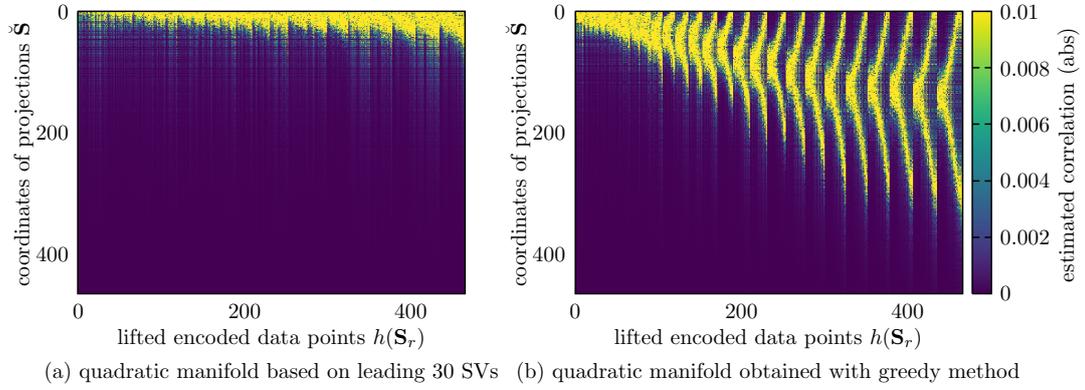}}
    \caption{Turbulent flow: The greedily selected basis matrix $\Vlin$ leads to a quadratic manifold so that the lifted encoded data points are well correlated to the coordinates corresponding to the later left-singular vectors, which is a necessary condition for quadratic manifolds to achieve higher accuracy than the corresponding linear approximation spaces.}
    \label{fig:channelflow_correlations}
\end{figure}

\section{Conclusions}\label{sec:Conc} Augmenting linear decoder functions with nonlinear correction terms given by feature maps can lead to higher accuracy than linear approximations alone; however, because the corrections are added to the decoder function, the feature maps are evaluated only at the encoded data points rather than the original, high-dimensional data points. In this work, we showed that linear best-approximations given by projections onto the principal components can lead to poor results in combination with correction terms because the data points encoded in the first few leading principal components can miss information that are important for the correction terms to be efficient. The greedy method introduced in this approach allows selecting principal components that are not necessarily ordered descending with respect to the  singular values. Numerical experiments demonstrate that an orders of magnitude higher accuracy can be achieved with the introduced greedy method and that the approach scales to data points with millions of dimensions.

Code is available at \url{https://github.com/Algopaul/greedy_quadratic_manifolds}.

\section*{Acknowledgements}
We thank Prakash Mohan (NREL) for sharing the turbulent flow data. This work was also supported in part through the NYU IT High Performance Computing resources, services, and staff expertise.

\bibliographystyle{myplain}
\bibliography{donotchange,references}

\appendix
\section{Data generation}

\subsection{Nonlinear advection-diffusion processes}
\label{app:burgers_setup} The data has been generated by numerically solving the viscous Burgers' equation, which is
\begin{align}
  \label{eq:burgers}
  \begin{split}
    \partial_t s(t, x) + s(t,x) \partial_x s(t, x) - \nu \partial_{xx}^2 s(t, x) &= 0, \\
    s(t, x) &= s_0(x),
  \end{split}
\end{align}
where $\nu=10^{-4}$. We imposed periodic boundary conditions in the spatial domain $[-1, 1)$ and used the time interval $[0, 1]$.
We collect data for varying initial conditions given by
\begin{align}
  s_0(x) =  0.3 \exp\left(-\mu^2 (x + 0.5)^2\right) + 1,
\end{align}
where we vary the sharpness of the spike in the interval $\mu \in [10, 15]$. To collect training data, we discretize the spatial domain into $n=5000$ degrees of freedom using a finite difference scheme. Then we solve the resulting ordinary differential equation using a Runge-Kutta method of order four. This generates 500 data points per computed solution trajectory. We set the parameter $\mu$ to the values  $\{10, 11.25, 13.75, 15\}$ and compute four trajectories to generate the train data set, which consequently consists of $2000$ data points. Moreover, to generate the validation and test data set, we compute another trajectory, where we set $\mu$ to $12.5$ and assign the data points in this additional trajectory to either the validation set or the test set, alternatingly.

Figure~\ref{fig:burgers_reconstruction} shows the approximation of a test data point obtained with the quadratic manifolds and the linear approximation. In agreement with the errors reported in Figure~\ref{fig:burgers_example}, the linear approximation and the quadratic manifold based on the leading $\nred$ singular vectors lead to comparable approximations in terms of error. In contrast, the proposed greedy approach constructs a quadratic manifold that leads to an approximation of the test data point that cannot be distinguished visually anymore from the original data point.

\begin{figure}
\resizebox{0.99\columnwidth}{!}{\input{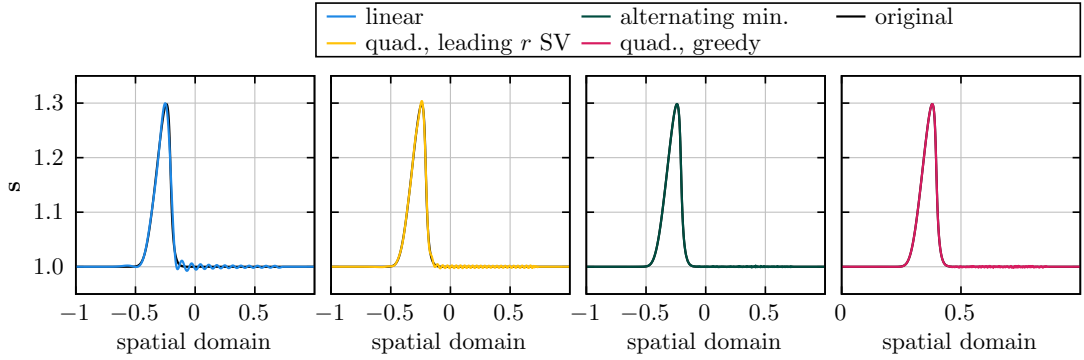}}
  \caption{Nonlinear advection-diffusion: The quadratic manifold obtained with the proposed greedy method provides an accurate approximation of the test data point.}
  \label{fig:burgers_reconstruction}
\end{figure}

\subsection{Hamiltonian interacting pulse signals}
\label{app:wave_setup}
The data was generated by numerically solving the acoustic wave equation over a two-dimensional spatial domain in Hamiltonian form with periodic boundary conditions in the spatial domain $[-4, 4)^2$,
\begin{align}
  \label{eq:acoustic_wave_equation}
  \begin{split}
    \partial_t \rho(t, x) 	&= -\nabla\cdot v(t, x)\,,\\
    \partial_t v(t, x) 		&= -\nabla\rho(t, x)\,, \\
    \rho(0, x) &= \rho_0(x)\,, \\
    v(t, 0) &= 0\,,
  \end{split}
\end{align}
where $\rho(t, x) \in \R$ denotes the density and $v(t,x) \in \R^2$ denotes the velocity field. We set the initial condition to
\begin{equation}
  \rho_0(x) = \exp\left(-(2\pi)^2 \left((x_1 - 2)^2 + (x_2 - 2)^2\right)\right),
\end{equation}
and $v(0, x)=0$. We use a finite difference scheme with 600 degrees of freedom in each spatial direction, which leads to a state-space dimension $\nfull=1\,080\,000$. We collect $1600$ solutions computed with the Runge-Kutta method of order 4 in the time-interval $[0,8]$.

\begin{figure}
    \centering
    \begin{tabular}{c}
  \resizebox{0.99\columnwidth}{!}{\input{./PlotSources/linear_transport_am_convergence.tex}} \\
  \resizebox{0.99\columnwidth}{!}{\input{./PlotSources/burgers_am_convergence.tex}} \\
    \end{tabular}
    \caption{Convergence behavior of the alternating minimization approach for the advecting wave and nonlinear advection-diffusion example with reduced dimension $\nred = 10$ and $\nred = 30$, respectively. The convergence is slow, which partially explains the high runtime of constructing quadratic manifolds with alternating minimization.} 
\label{fig:Appendix:am_convergence}
\end{figure}

\section{Alternating minimization}
\label{appx:AlternatingMin}

The alternating minimization algorithm presented in~\cite{GeelenBW2023Learning} consists of the three following alternating steps.
\begin{enumerate}
   \item Solve orthogonal Procrustes problem
      \begin{align}
      \label{eq:procrust_prob}
        [\Vlin, \Vnonlinhat] = \argmin\limits_{\mathbf{X} \in \R^{\nfull \times (\nred + \vbarconsider)}} 
        {\frac12
        \left\|
        \snapshots - \mathbf{X} \begin{bmatrix} \snapshots_r \\ \maniXi \end{bmatrix}
        \right\|}_F^2,
        \text{ such that } \mathbf{X}^\top \mathbf{X} = \identity,
      \end{align}
      where $\snapshots_r$ denotes the reduced data points.
    \item Compute $\maniXi$ by solving a least squares problem
      \begin{equation}
      \label{eq:xi_lstsq}
        \maniXi = \argmin\limits_{\mathbf{X} \in \R^{\vbarconsider \times \nredmod}}
        \left(
          \frac12 {\left\|
            \featuremap(\snapshots_r)^\top \mathbf{X}^\top - (\snapshots-\Vlin\snapshots_r)^\top \Vnonlinhat
          \right\|}_{F}^2 + \frac{\gamma}2 {\left\|\mathbf{X}\right\|}_F^2
        \right).
      \end{equation}
    \item Compute the reduced data points by solving the nonlinear optimization problem
      \begin{align}
        \label{eq:optprob_am}
        \snapshots_r = \argmin\limits_{X \in \R^{\nred \times \nsnapshots}} \frac12
        \sum\limits_{j=1}^k
        \left\|
        \fullstatei{j}-\begin{bmatrix}
          \Vlin & \Vnonlinhat
        \end{bmatrix}
        \begin{bmatrix}
          \redstatei{j} \\ \maniXi \featuremap(\redstatei{j})
        \end{bmatrix}
        \right\|.
      \end{align}
\end{enumerate}

In the first step, the hyper-parameter $\vbarconsider$ is introduced that sets the number of columns in $\Vnonlinhat$. A larger value of $\vbarconsider$ leads to a higher runtime but when $\vbarconsider$ is chosen too small, a decrease in accuracy can be noted. This is because by solving~\eqref{eq:procrust_prob} and subsequently~\eqref{eq:xi_lstsq} truncates the singular value decomposition of $\snapshots$ to its first $\vbarconsider$ components. For a fair comparison, we choose $\vbarconsider=\nconsider$ in our experiments. In the second step, the hyper-parameter is just the regularization parameter $\gamma$, which is also present in the method in~\cite{GeelenWW2023Operator} and in our method. The third step requires the solution of a nonlinear optimization problem. Here we follow the recommendation from~\cite{GeelenBW2023Learning} and use a Levenberg-Marquardt algorithm. More precisely, we use the setup from the supplementary code from~\cite{GeelenBWW2024Learning} available at~\url{https://github.com/geelenr/nl_manifolds/blob/main/nl_manifolds.ipynb}, which is based on \texttt{scipy.opt.least\_squares}. We set the option \texttt{max\_nfev} (which limits the number of objective function evaluations) to 1600 to obtain an acceptable runtime; recall that in one of our numerical experiments the runtime is already eleven days. 

Additionally to the hyper-parameters introduced in these three steps, the alternating minimization approach requires setting a maximum number of alternating minimization iterations as well as a convergence tolerance for the criterion~\cite[Eq. 16]{GeelenBW2023Learning}. Setting the convergence tolerance and the maximum number of iterations is a delicate issue because the alternating minimization approach has a sublinear convergence rate (see Figure~\ref{fig:Appendix:am_convergence}) so for later iterations the additional runtime has diminishing returns. Moreover, the convergence criterion varies by orders of magnitude between the different examples. This has led us to choose a small convergence tolerance of $10^{-12}$ to avoid under-reporting the accuracy of the alternating minimization scheme and additionally limit the runtime by setting the maximum number of iterations to $15\times\nred$ to keep the experiments computationally tractable.

\end{document}